\documentclass[12pt]{amsart}
\usepackage[margin=1.in, centering]{geometry}

\usepackage{amssymb,latexsym,amsmath,extarrows, mathrsfs, amsthm}
\usepackage{bbm}
\usepackage{amsmath}
\usepackage{graphicx}

\usepackage[colorlinks, linkcolor=blue]{hyperref}

\usepackage{graphicx} 

\usepackage{color}
\usepackage{wasysym}
\usepackage{float}

\newtheorem{theorem}{Theorem}[section]
\newtheorem{lemma}[theorem]{Lemma}
\newtheorem{proposition}[theorem]{Proposition}

\newtheorem{conjecture}[theorem]{Conjecture}

\theoremstyle{remark}

\title[Sign changes of Kloosterman sums]{Sign changes of Kloosterman sums with moduli having at most six prime factors}

\author[Tianping Zhang]{Tianping Zhang$^{1,2}$}
\address{1. School of Mathematics and Statistics, Shaanxi Normal University, Xi'an, 710119, Shaanxi, P. R. China}
\address{2. Research Center for Number Theory and Its Applications, Northwest University, Xi'an, 710127, Shaanxi, P. R. China}
\email{tpzhang@snnu.edu.cn}

\author[Mingxuan Zhong]{Mingxuan Zhong$^{1,*}$}
\address{Mingxuan Zhong is the corresponding author}
\email{zhong@snnu.edu.cn}

\allowdisplaybreaks

\date{}

\begin{document}
\def \LLN {\prec \hskip-6pt \prec}

\begin{abstract}
We prove that the Kloosterman sum $\text{Kl}(1,q)$ changes sign infinitely many times, as $q\rightarrow +\infty$ with at most six prime factors. As a consequence, our result improved the best known result of Xi(IMRN, 2022). The novelty of our method comes from introducing a new truncated divisor function whose selection depends on the number of prime factors of the variable, through which Kloosterman sum is controlled good enough. Our arguments contain the Selberg sieve method, spectral theory and distribution of Kloosterman sums along with previous nice works by Fouvry, Matom\"{a}ki, Michel, Sivak-Fischler and Xi.
\bigskip

\textbf{Keywords} Kloosterman sum; Sign change; Selberg sieve method; Spectral theory; Equidistribution
\bigskip

\textbf{MSC(2020)} 11L05, 11N36(11N75, 11L20, 26D15)

\end{abstract}

\maketitle

\section{Introduction}

\subsection{Background}

Let $q$ be a fixed positive integer. For arbitrary integers $m$, the classical Kloosterman sums are defined by
$$
\text{Kl}(m;q)=\frac{1}{\sqrt{q}}\sum_{b(q)\atop (b,q)=1}e\left(\frac{mb+\overline{b}}{q}\right),
$$
where $e(x)=e^{2\pi ix}$. Kloosterman sums have originated from Poincar\'{e} \cite{Poincare1911} and Kloosterman \cite{Kloosterman1927}. Scholars are concerned about whether there exists some specific explicit expression or asymptotic formula? Unfortunately, these results are beyond our capabilities for now.

Kloosterman sums are one of the central topics in analytic number theory and related to numerous applications in Diophantine equations and automorphic forms. A well-known estimate from Weil \cite{Weil1948} for the individual Kloosterman sum is
$$
\left|\text{Kl}(m;p)\right|\leq 2(m,p)^{1/2},
$$
in general
$$
\left|\text{Kl}(m;p)\right|\leq (m,q)^{1/2}\tau(q),
$$
and the factor $\tau(q)$ above can be modified to $2^{\omega(q)}$ for $32\nmid q$ by Estermann \cite{Estermann1961} ($\omega(q)$ represents the number of different prime factors of $q$).

The equidistributions of Kloosterman sums are proposed by Katz \cite{Katz1980} and analogized by the Sato-Tate conjecture of elliptic curves, which we state as

\begin{conjecture}\label{Sato-Tate conj}
For each $f\in\mathcal{C}([0,\pi])$ and non-zero integer $a$, we have
$$
\lim_{x\rightarrow +\infty}\frac{1}{\pi(x)}\sum_{p\leq x}f\left(\theta_{p}(a)\right)=\frac{2}{\pi}\int_{0}^{\pi}f(\theta)\sin^2{\theta}d\theta,
$$
where $\theta_{p}(a)\in [0,\pi]$ is so-called Kloosterman sum angle defined as
$$
2\cos{\theta_{p}(a)}:=\text{Kl}(a;p).
$$
\end{conjecture}

Later, we usually refer to the above conjecture as the ``horizontal" Sato-Tate conjecture. Katz inferred that the angles $\theta_{p}(a)$(for each $a$) are equidistributed, as $x\rightarrow +\infty$, with respect to the Sato-Tate measure $$\mu_{\text{ST}}=\frac{2}{\pi}\sin^2{\theta}d\theta.$$ In \cite{Katz1988}, Katz showed that the numbers
$$
\left\{\theta_{p}(a):1\leq a<p\right\}
$$
equidistribute with respect to the same measure, as $p\rightarrow +\infty$, and we call it the ``vertical" Sato-Tate law. Furthermore, a direct corollary of the conjecture is that Kloosterman sums change signs infinitely many times as $p\rightarrow +\infty$.

At present, we know very little about the ``horizontal" Sato-Tate conjecture, but we still believe in its correctness. Part of the evidence is due to the numerical calculations of Kloosterman sums are highly consistent with ``horizontal" Sato-Tate conjecture \cite{Michel2006}, and the other part is due to certain exponential sums such as cubic Gaussian sums and  Sali\'{e} sums do satisfy ``horizontal" Sato-Tate conjecture. About the relevant applications and progress of Kloosterman sums can refer to \cite{Michel2006}.

In this paper, we mainly concern about a slightly weaker problem which is the sign changes of Kloosterman sums. Fouvry and Michel \cite{Fouvry2003, Fouvry2007} pioneeringly proved that
$$
\left|\left\{X<q\leq 2X: \text{Kl}(1;q)\gtrless 0, \omega(q)\leq 23, q \ \text{square-free}\right\}\right|\gg\frac{X}{\log X}.
$$
Subsequently, this result was improved in a whole series of papers by Sivak-Fischler \cite{Sivak-Fischler2007, Sivak-Fischler2009}, Matom\"{a}ki \cite{Matomaki2011}, Xi \cite{Xi2015, Xi2018, Xi2022} by replacing 23 with 18, 15, 10 and 7.

Recently, Drappeau and Maynard \cite{Drappeau2019} proved that under the condition of the existence of Landau-Siegel zero, 23 can be reduced to 2.

\subsection{Our result}

Our result is stated as the following.

\begin{theorem}\label{thm: main theorem}
Let $q$ be any square-free number with $\omega(q)\leq6$. Then $\text{Kl}(1;q)$ changes sign infinitely many times, as $q\rightarrow +\infty$.
\end{theorem}

It seems that this conclusion is very difficult to be improved. The main reasons are to be stated in Section \ref{improve}.

{\bf Notations } Before we start, we need to give some notations:

$\bullet$ $U\ll V$ or $U=O(V)$, means $|U|\leq cV$ for some constant $c>0$.

$\bullet$ If not specified, we use $a\ (q)$ to represent $a(\bmod\ q)$ and $\overline{a}$ to represent the multiplicative inverse modulo $q$.

$\bullet$ $\mathcal{L}=\log X$.

$\bullet$ $\binom{a}{b}$ is the binomial coefficient.

$\bullet$ $\widetilde{g}$ is the Mellin transform of $g$.

$\bullet$ $\#\left\{\mathcal{A}\right\}$ means the cardinality of the set $\mathcal{A}$.

$\bullet$ $\lVert\boldsymbol{a}\rVert$ represents $l_{2}$-norm.

$\bullet$ $\displaystyle\mathop{{\sum}^*}_{a(q)}$ represents the summation of $1\leq a\leq q$  with $(a,q)=1$.

\section{Preliminary}\label{sub: Pre.}
\subsection{Truncated divisor function}\label{selection}

Define a new truncated divisor function
$$
\tau(n;\alpha,\beta)=\sum_{n=dl\atop d\in\mathcal{B}(n)}\alpha^{\omega(d)}\beta^{\omega(l)}
$$
with
\begin{align}\label{subset select}
\mathcal{B}(n)=\left\{d\in\mathbb{Z}^{+}: d\mid n, \omega(d)=3, \omega\left(\frac{n}{d}\right)\ge 3, d\leq n^{\frac{1}{2}}\right\}
\end{align}
for square-free numbers $n$.

For our purposes, we need to select a subset from
$$
\left\{d\in\mathbb{Z}^{+}: d\mid n, d\leq n^{\frac{1}{2}}\right\}.
$$
We hope that the number of elements in this subset is appropriately small. After a long time of checking, we select \eqref{subset select} as the very required subset. What we need to emphasize is that this selection is quite different from the previous one in that it depends further on the number of prime factors of the variable, not just the size and range of the variable.

The reason for this selection will be stated in Section \ref{improve}.

\subsection{Outline of the proof}

Our purpose is to show that
$$
\left|\left\{X<q\leq 2X: \text{Kl}(1;q)\gtrless 0, \omega(q)\leq6, q \ \text{square-free}\right\}\right|\ge c_{0}\frac{X}{\log X},
$$
where $c_{0}>0$ is a constant. Following much the same way in \cite{Xi2015, Xi2018, Xi2022}, let $g(x)$ be a fixed smooth function supported on $[1,2]$, $\left(\lambda_{d}\right)$ be the Selberg sieve weight satisfying
$$
\begin{cases}
\lambda_{1}=1, \\
\left|\lambda_{d}\right|\leq 1,\\
\lambda_{d}=0, \ \text{if}\ d>\sqrt{D},
\end{cases}
$$
where $D$ is the level of Selberg sieve. More specifically, we set $D=X^{1/2-\epsilon}$, and
$$
\lambda_{d}=\mu(d)F\left(\frac{\log\left(\sqrt{D}/d\right)}{\log\left(\sqrt{D}\right)}\right),
$$
where $F$ is a fixed smooth function supported on $[0,1]$ and vanishes at 0 to a suitable order (related to $k$ below, i.e. the dimension of the sieve problem). This choice of the order of $F$ is optimal, see \cite{Granville2021} (Page 1090) for further discussion.

In this paper, we specifically take
$$
F(x)=x^4(a_{0}+a_{1}x+a_{2}x^2+a_{3}x^3+a_{4}x^4),
$$
with
$$
(a_{0},a_{1},a_{2},a_{3},a_{4})\approx \left(\frac{100396}{53901},-\frac{17284}{13475},\frac{76486}{134753},-\frac{33241}{188654},\frac{10836}{377308}\right).
$$

Define
\begin{align}\label{original purpose}
R^{\pm}(X):=\sum_{\left(n,\Pi_{\varepsilon}\right)=1}g\left(\frac{n}{X}\right)\mu^{2}(n)\left(\left|\text{Kl}(1;n)\right|\pm\text{Kl}(1;n)\right)\left\{\rho-\tau\left(n;\alpha,\beta\right)\right\}\left(\sum_{d\mid n}\lambda_{d}\right)^2,
\end{align}
where $\rho$, $\alpha$, $\beta$ are fixed positive real numbers to be chosen later, $\displaystyle\Pi_{\varepsilon}=\prod_{p<X^{\varepsilon}}p$ is the prime factor product not exceeding $X^{\varepsilon}$ with a suitably small constant $\varepsilon>0$.

Now we hope $R^{\pm}(X)$ has a lower bound larger than 0, as shown below
$$R^{\pm}(X)\ge \rho R_{1}(X)-2R_{2}(X)>0,$$
where
$$
R_{1}(X):=\sum_{\left(n,\Pi_{\varepsilon}\right)=1}g\left(\frac{n}{X}\right)\mu^{2}(n)\left(\left|\text{Kl}(1;n)\right|\pm\text{Kl}(1;n)\right)\left(\sum_{d\mid n}\lambda_{d}\right)^2,
$$
$$
R_{2}(X):=\sum_{\left(n,\Pi_{\varepsilon}\right)=1}g\left(\frac{n}{X}\right)\mu^{2}(n)\left|\text{Kl}(1;n)\right|\tau\left(n;\alpha,\beta\right)\left(\sum_{d\mid n}\lambda_{d}\right)^2.
$$

The lower bound for $R_{1}(X)$ can be divided into two parts, one of which can be obtained in \cite{Fouvry2007}, \cite{Matomaki2011}, \cite{Sivak-Fischler2009} and \cite{Xi2015}, and the other part can be processed in a similar way as in \cite{Xi2018}.

For $R_{2}(X)$, we will further analyze this newly introduced divisor function. Part of the process can be handled in a similar way as in \cite{Xi2022}. Building upon \eqref{subset select}, we need to analyze some special forms of Lemma \ref{lem: BDH type for |KL|}, which we will discuss in detail later.

To prove Theorem \ref{thm: main theorem}, we need the following conclusions.

\begin{proposition}\label{lower bound for |KL|}
For any sufficiently large number $X$, we have
$$
R_{1}(X)\ge (1+o(1))0.76235\times\widetilde{g}(1)\frac{X}{\log X},
$$
where $g(X)$ is a fixed smooth function defined above.
\end{proposition}

\begin{proposition}\label{upper bound for |KL|}
For any sufficiently large number $X$, suppose
$$
\alpha>0,\ \beta>0,
$$
and choose $\alpha=\alpha_{0}>0$, $\beta=1$, specifically. Then we have
$$
R_{2}(X)\leq (1+o(1))6.27044\alpha_{0}^3\times\widetilde{g}(1)\frac{X}{\log X}.
$$
\end{proposition}

\section{Some Lemmas}

The following lemmas are needed.

\begin{lemma} \label{lem: equidistribution for |KL|}
For a sufficiently large real number $X$, define
$$
H(X)=\sum_{n}g\left(\frac{n}{X}\right)\mu^{2}(n)\left|\text{Kl}(1;n)\right|\left(\sum_{d\mid n}\lambda_{d}\right)^2.
$$
Thus we have
$$
H(X)\ge \widetilde{g}(1)\sum_{2\leq i\leq 5}2^i C_{i}A_{i}(F)\cdot\frac{X}{\log X}(1+o(1)),
$$
where $C_{2}=0.11109$, $C_{3}=0.03557$, $C_{4}=0.01184$, $C_{5}=0.00396$,
$$
A_{i}(F)=\mathop{\int\ldots\int}_{R_{i}}\frac{L_{i}^2\left(F;X^{1-\alpha_{2}-\cdots-\alpha_{i}},X^{\alpha_{2}},\ldots,X^{\alpha_{i}}\right)}{\alpha_{2}\cdots\alpha_{i}(1-\alpha_{2}-\cdots-\alpha_{i})}\text{d}\alpha_{2}\cdots\text{d}\alpha_{i},
$$
$$
L_{i}(F;\alpha_{1},\alpha_{2},\ldots,\alpha_{i})=\sum_{\mathcal{A}\subseteq \left\{\alpha_{1},\alpha_{2},\ldots,\alpha_{i}\right\}\atop \sum_{\alpha\in\mathcal{A}}\alpha<\frac{1}{4}}(-1)^{|\mathcal{A}|}F\left(1-4\sum_{\alpha\in\mathcal{A}}\alpha\right),
$$
and
\begin{align}
R_{2}:=&\left\{\alpha_{2}\in[\eta,1):\left(\frac{3}{4}+\eta\right)(1-\alpha_{2})<\alpha_{2}<\frac{1}{2}\right\},\ \eta:=10^{-2016},\notag\\
R_{3}:=&\left\{(\alpha_{2},\alpha_{3})\in [\eta,1)^2:\frac{1}{2}(1-\alpha_{2}-\alpha_{3})<\alpha_{2},\ \alpha_{3}<\alpha_{2}<1-\alpha_{2}-\alpha_{3}  \right\},\notag\\
R_{4}:=&\left\{(\alpha_{2},\alpha_{3},\alpha_{4})\in [\eta,1)^3:\frac{1}{2}(1-\alpha_{2}-\alpha_{3}-\alpha_{4})<\alpha_{2}+\alpha_{3}\right\}\notag\\
&\cap\left\{(\alpha_{2},\alpha_{3},\alpha_{4})\in [\eta,1)^3:\alpha_{4}<\alpha_{3}<\alpha_{2}<1-\alpha_{2}-\alpha_{3}-\alpha_{4}\right\},\notag\\
R_{5}:=&\left\{(\alpha_{2},\alpha_{3},\alpha_{4},\alpha_{5})\in [\eta,1)^4:\frac{1}{2}(1-\alpha_{2}-\alpha_{3}-\alpha_{4}-\alpha_{5})<\alpha_{2}+\alpha_{3}+\alpha_{4}\right\}\notag\\
&\cap\left\{(\alpha_{2},\alpha_{3},\alpha_{4},\alpha_{5})\in [\eta,1)^4:\frac{1}{2}(\alpha_{3}+\alpha_{4}+\alpha_{5})<\alpha_{2}\right\}\notag\\
&\cap\left\{(\alpha_{2},\alpha_{3},\alpha_{4},\alpha_{5})\in [\eta,1)^4:\alpha_{5}<\alpha_{4}<\alpha_{3}<\alpha_{2}<1-\alpha_{2}-\alpha_{3}-\alpha_{4}-\alpha_{5}\right\}.\notag
\end{align}
\end{lemma}

\begin{proof}
See Proposition 2.1 in \cite{Xi2018}.
\end{proof}

\begin{lemma}\label{lem: trace formula for |KL|}
For any $A>0$, there exists some $B=B(A)>0$ such that
$$
\sum_{q\leq \sqrt{X}\mathcal{L}^{-B}}3^{\omega(q)}\left|\sum_{n\equiv 0(q)}\mu^2(n)g\left(\frac{n}{X}\right)\text{Kl}(1;n)\right|\ll X\mathcal{L}^{-A},
$$
where the implied constant depends on $A$ and $g$.
\end{lemma}

\begin{proof}
See Lemma 2.1 in \cite{Xi2015}.
\end{proof}

Before giving the following lemma, we first recall the ``Siegel-Walfisz" condition. Let $\boldsymbol{v}=(v_{n})$ be an arithmetic function. For any positive integers $w$, $d$ with $(w,a)=1$, $a\neq 0$, if
$$
\sum_{n\leq X\atop{n\equiv a(w)\atop (n,d)=1}}v_{n}-\frac{1}{\phi(w)}\sum_{n\leq X\atop (n,dw)=1}v_{n}\ll \lVert\boldsymbol{v}\rVert X^{\frac{1}{2}}\tau(d)^{B}\mathcal{L}^{-A}
$$
holds for some constant $B>0$ and any $A>0$ with an implied constant in $\ll$ depending only on $A$, then we say $\boldsymbol{v}$ satisfies the ``Siegel-Walfisz" condition \cite{Xi2020}.

\begin{lemma}\label{lem: BDH type for |KL|}
Let $\boldsymbol{\alpha}=(\alpha_{m})$ be a complex coefficient with support in $[M,2M]$ satisfying the above ``Siegel-Walfisz" condition. For square-free $q\ge 1$, put
$$
\varepsilon(\boldsymbol{\alpha},\boldsymbol{\beta},\boldsymbol{\gamma}_{q};q)=\frac{1}{\sqrt{q}}\mathop{{\sum}^*}_{r( q)}\left|\text{Kl}\left(\overline{r}^2;q\right)\right|\left(\mathop{\sum\sum}_{mn\equiv r(q)}\alpha_{m}\beta_{n}\gamma_{n,q}-\frac{1}{\phi(q)}\mathop{\sum\sum}_{(mn,q)=1}\alpha_{m}\beta_{n}\gamma_{n,q}\right),
$$
where $\boldsymbol{\beta}=(\beta_{n})$, $\boldsymbol{\gamma}_{q}=(\gamma_{n,q})$ are complex coefficients with support on $[N,2N]$ with $\lVert\boldsymbol{\gamma}_{q}\rVert_{\infty}\leq (\tau(q)\log 2q)^C$ for some $C>0$.

Let $r\ge 1$ and $M\ge N$. For any $A>0$, there exists some constant $B=B(A,C)>0$, such that
$$
\sum_{q\leq Q}\mu^2(q)\tau(q)^r \left|\varepsilon(\boldsymbol{\alpha},\boldsymbol{\beta},\boldsymbol{\gamma}_{q};q)\right|\ll \lVert\boldsymbol{\alpha}\rVert\lVert\boldsymbol{\beta}\rVert Q(MN)^{\frac{1}{2}}(\log MN)^{-A}
$$
for $Q\leq MN(\log MN)^{-B}$, where the implied constant depends only on $A$ and $r$.
\end{lemma}

\begin{proof}
See Lemma 4.1 in \cite{Xi2020}.
\end{proof}

\begin{lemma}\label{lem: equidistribution of KL}
For sufficiently large prime $p$, we have
$$
\mathop{{\sum}^*}_{a(p)}\left|\text{Kl}\left(\overline{a}^2;p\right)\right|=\frac{8}{3\pi}p+O(\sqrt{p}\log p).
$$
\end{lemma}

\begin{proof}
See Lemma 3.3 in \cite{Xi2020}.
\end{proof}

\begin{lemma}\label{lem: Selberg sieve for divisor function}
Let $g$ and $(\lambda_{d})$ be given above. Let $k$ be a fixed positive integer and $F$ has a zero of order at least $k$ at 0. Then
$$
\sum_{n}g\left(\frac{n}{X}\right)\mu^2(n)k^{\omega(n)}\left(\sum_{d\mid n}\lambda_{d}\right)^2=4\widetilde{g}(1)c(k,F)\frac{X}{\log X}(1+o(1)),
$$
where $c(k,F)$ is defined by
$$
c(k,F)=\sum_{j=1}^{k}\frac{1}{\Gamma(j)^2}\binom{k}{j}\int_{0}^{1}F^{(j)}(x)^2(1-x)^{j-1}\left(3+x\right)^{j-1}\text{d}x.
$$
\end{lemma}

\begin{proof}
See Proposition 4.1 in \cite{Xi2018}.
\end{proof}

Before presenting the following lemma, let's first establish some definitions. Define a smooth function $J(x)$ and assume it has a Taylor expansion
$$
J(x):=\sum_{n\ge 0}\frac{a_{n}}{n!}x^n
$$
with $J(0)=0$. For $N>1$ and $s\in\mathbb{C}\backslash \left\{0\right\}$, we further define
$$
\check{J}_{N}(s):=\sum_{n\ge 0}\frac{a_{n}}{(s\log N)^n}.
$$

\begin{lemma}\label{lem: Mellin trans for F}
Let $N>1$ be not an integer. For any coefficient $\boldsymbol{y}=\left(y_{n}\right)$ with $y_{n}\ll\tau(n)^{O(1)}(\log n)^{O(1)}$, we have
$$
\sum_{n\leq N}y_{n}J\left(\frac{\log(N/n)}{\log N}\right)=\frac{1}{2\pi i}\int_{2-i\infty}^{2+i\infty}\check{J}_{N}(s)Y(s)\frac{N^s}{s}\text{d}s,
$$
where
$$
Y(s)=\sum_{n\ge 0}\frac{y_{n}}{n^s},\ \Re s>1.
$$
\end{lemma}

\begin{proof} See Lemma 3.4 in \cite{Xi2018}.\end{proof}

\begin{lemma}\label{lem: double residue}
Let $v$, $v_{1}$, $v_{2}$ be fixed positive integers, $m$ be fixed even positive integer, $M> 1$ be any real number. Let $P(x)$ and $Q(x)$ be two smooth functions with zeros at 0 of orders at least $v_{1}+m/2$, $v_{2}+m/2$, respectively. Let $Z(s_{1},s_{2})$ be holomorphic in the right half plane containing a neighborhood of $(0,0)$ with
$$
\left(\frac{\text{d}^n}{\text{d}s^n}Z(s,s\xi)\right)_{s=0}=0, \ 0\leq n<m,
$$
and
$$
\left(\frac{\text{d}^m}{\text{d}s^m}Z(s,s\xi)\right)_{s=0}=C_{0}\xi^{\frac{m}{2}}Z^{(m)}(0,0)\neq 0,
$$
where $|\xi|=2$ is a circle in $\mathbb{C}$, $C_{0}$ is a constant. Put
$$
R:=\mathop{\text{Res}}_{{(s_{1},s_{2})=(0,0)}}\check{P}_{M}(s_{1})\check{Q}_{M}(s_{2})Z(s_{1},s_{2})s_{1}^{v_{1}-1}s_{2}^{v_{2}-1}\frac{M^{s_{1}+s_{2}}}{(s_{1}+s_{2})^v}.
$$
Thus we have
$$
R=(1+o(1))C_{0}Z^{(m)}(0,0)\frac{(\log M)^{v-v_{1}-v_{2}-m}}{\Gamma(v)m!}\int_{0}^{1}P^{(v_{1}+m/2)}(x)Q^{(v_{2}+m/2)}(x)(1-x)^{v-1}\text{d}x.
$$
\end{lemma}

\begin{proof}We follow the proof of Lemma A.2 in \cite{Xi2018} and use Goldston-Motohashi-Pintz-Y\i ld\i r\i m method in \cite{Goldston2006}. Assume
$$
P(x):=\sum_{k\ge 0}\frac{a_{k}}{k!}x^k,\ Q(x):=\sum_{k\ge 0}\frac{b_{k}}{k!}x^k.
$$
For $M>1$ and $s\in\mathbb{C}\backslash \left\{0\right\}$, define
$$
\check{P}_{M}(s)=\sum_{k\ge 0}\frac{a_{k}}{(s\log M)^k},\ \check{Q}_{M}(s)=\sum_{k\ge 0}\frac{b_{k}}{(s\log M)^k}.
$$
Therefore, we have
\begin{align}\label{double residue equ 1}
R=\sum_{k_{1}\ge v_{1}}\sum_{k_{2}\ge v_{2}}\frac{a_{k_{1}}b_{k_{2}}}{(\log M)^{k_{1}+k_{2}}}\mathop{\text{Res}}_{{(s_{1},s_{2})=(0,0)}}Z(s_{1},s_{2})\frac{M^{s_{1}+s_{2}}}{s_{1}^{k_{1}-v_{1}+1}s_{2}^{k_{2}-v_{2}+1}(s_{1}+s_{2})^v},
\end{align}
and the residue therein is equal to
$$
\frac{1}{(2\pi i)^2}\int_{C_{2}}\int_{C_{1}}\frac{Z(s_{1},s_{2})M^{s_{1}+s_{2}}}{s_{1}^{k_{1}-v_{1}+1}s_{2}^{k_{2}-v_{2}+1}(s_{1}+s_{2})^v}\text{d}s_{1}\text{d}s_{2},
$$
where $C_{1}$, $C_{2}$ are the circles $|s_{1}|=\rho$ and $|s_{2}|=2\rho$ with a small $\rho>0$. We write $s_{1}=s$, $s_{2}=s\xi$, then the double integral can be rewritten as
\begin{align}\label{double residue equ 2}
\frac{1}{(2\pi i)^2}\int_{C_{3}}\int_{C_{1}}\frac{Z(s,s\xi)M^{s(1+\xi)}}{s^{v+k_{1}-v_{1}+k_{2}-v_{2}+1}\xi^{k_{2}-v_{2}+1}(\xi+1)^v}\text{d}s\text{d}\xi,
\end{align}
where $C_{3}$ is the circle $|\xi|=2$.

The integral of $s$ is equal to
$$
\frac{1}{(v+k_{1}-v_{1}+k_{2}-v_{2})!}\left(\frac{\text{d}^{v+k_{1}-v_{1}+k_{2}-v_{2}}}{\text{d}s^{v+k_{1}-v_{1}+k_{2}-v_{2}}}Z(s,s\xi)M^{s(1+\xi)}\right)_{s=0}.
$$
From the condition of Lemma \ref{lem: double residue}, this can be expanded as
\begin{align}
\sum_{i=m}^{v+k_{1}-v_{1}+k_{2}-v_{2}}\binom{v+k_{1}-v_{1}+k_{2}-v_{2}}{i}&\left(\frac{\text{d}^i}{\text{d}s^i}Z(s,s\xi)\right)_{s=0}\notag\\
&\times(\xi+1)^{v+k_{1}-v_{1}+k_{2}-v_{2}-i}(\log M)^{v+k_{1}-v_{1}+k_{2}-v_{2}-i}.\notag
\end{align}
We only need to consider the term with the quantity $(\log M)^{v+k_{1}-v_{1}+k_{2}-v_{2}-m}$, since the term containing $(\log M)^{v+k_{1}-v_{1}+k_{2}-v_{2}-m}$ has the highest power with respect to $M$. Substituting the expression above into \eqref{double residue equ 2}, we derive
\begin{align}\label{double residue equ 3}
\frac{C_{0}}{(v+k_{1}-v_{1}+k_{2}-v_{2})!}\binom{v+k_{1}-v_{1}+k_{2}-v_{2}}{m}&Z^{(m)}(0,0)(\log M)^{v+k_{1}-v_{1}+k_{2}-v_{2}-m}\notag\\
&\times \frac{1}{2\pi i}\int_{C_{3}}\frac{(\xi+1)^{k_{1}-v_{1}+k_{2}-v_{2}-m}}{\xi^{k_{2}-v_{2}-m/2+1}}\text{d}\xi.
\end{align}

For the integral in \eqref{double residue equ 3}, we will  deal with the following four cases as
$$
\begin{cases}
&k_{1}-v_{1}\ge \frac{m}{2}\\
&k_{2}-v_{2}\ge \frac{m}{2},
\end{cases}
\
\begin{cases}
&k_{1}-v_{1}< \frac{m}{2}\\
&k_{2}-v_{2}< \frac{m}{2},
\end{cases}
\
\begin{cases}
&k_{2}-v_{2}\ge \frac{m}{2}\\
&k_{1}-v_{1}< \frac{m}{2},
\end{cases}
\
\begin{cases}
&k_{2}-v_{2}<\frac{m}{2}\\
&k_{1}-v_{1}\ge \frac{m}{2}.
\end{cases}
$$
In the last three cases, the integral evaluates to zero. When $k_{1}-v_{1}< \frac{m}{2}$ and $k_{2}-v_{2}< \frac{m}{2}$, only one pole $\xi=-1$ of order $m-(k_{1}-v_{1})-(k_{2}-v_{2})$ is in the region $\left\{\xi: |\xi|<2\right\}$. Then Cauchy's residue theorem implies that
$$
\frac{1}{2\pi i}\int_{C_{3}}\frac{\xi^{m/2-(k_{2}-v_{2})-1}}{(\xi+1)^{m-(k_{1}-v_{1})-(k_{2}-v_{2})}}\text{d}\xi=\mathop{\text{Res}}_{\xi=-1}\frac{\xi^{m/2-(k_{2}-v_{2})-1}}{(\xi+1)^{m-(k_{1}-v_{1})-(k_{2}-v_{2})}}=0.
$$

When $k_{1}-v_{1}< \frac{m}{2}$ and $k_{2}-v_{2}\ge \frac{m}{2}$, we divide this case into two subcases, namely
$$
\begin{cases}
&k_{2}\ge v_{2}+\frac{m}{2}+\left(v_{1}+\frac{m}{2}-k_{1}\right)\\
&k_{1}<v_{1}+ \frac{m}{2},
\end{cases}
\
\begin{cases}
&v_{2}+\frac{m}{2}\leq k_{2}<v_{2}+\frac{m}{2}+\left(v_{1}+\frac{m}{2}-k_{1}\right)\\
&k_{1}<v_{1}+ \frac{m}{2}.
\end{cases}
$$
For the first one, we see the integral is equal to
$$
\frac{1}{2\pi i}\int_{C_{3}}\frac{(\xi+1)^{k_{1}+k_{2}-v_{1}-v_{2}-m}}{\xi^{k_{2}-v_{2}-m/2+1}}\text{d}\xi.
$$
Only one pole $\xi=0$ of order $k_{2}-v_{2}-m/2+1$ is in the region $\left\{\xi: |\xi|<2\right\}$. Thus Cauchy's residue theorem implies that
$$
\frac{1}{2\pi i}\int_{C_{3}}\frac{(\xi+1)^{k_{1}+k_{2}-v_{1}-v_{2}-m}}{\xi^{k_{2}-v_{2}-m/2+1}}\text{d}\xi=\mathop{\text{Res}}_{\xi=0}\frac{(\xi+1)^{k_{1}+k_{2}-v_{1}-v_{2}-m}}{\xi^{k_{2}-v_{2}-m/2+1}}=0.
$$

For the second one, the integral is
$$
\frac{1}{2\pi i}\int_{C_{3}}\frac{1}{(\xi+1)^{v_{1}+v_{2}+m-k_{1}-k_{2}}\xi^{k_{2}-v_{2}-m/2+1}}\text{d}\xi,
$$
and two poles $\xi=-1$ and $\xi=0$ are in the region $\left\{\xi: |\xi|<2\right\}$. Surprisingly, in this case, we can obtain
$$
\mathop{\text{Res}}_{\xi=-1}\frac{1}{(\xi+1)^{v_{1}+v_{2}+m-k_{1}-k_{2}}\xi^{k_{2}-v_{2}-m/2+1}}=-\mathop{\text{Res}}_{\xi=0}\frac{1}{(\xi+1)^{v_{1}+v_{2}+m-k_{1}-k_{2}}\xi^{k_{2}-v_{2}-m/2+1}},
$$
thus the integral evaluates to zero.

The last case can be handled similarly to the third one. For simplicity, it is advantageous to interchange the roles of $s_{1}$ and $s_{2}$ (i.e. let $|s_{1}|=2\rho$, $|s_{2}|=\rho$) in this analysis.

Combining the above results, only the case
$$
\begin{cases}
&k_{1}-v_{1}\ge \frac{m}{2}\\
&k_{2}-v_{2}\ge \frac{m}{2}
\end{cases}
$$
remains to be dealt with. The integral contributes $\binom{k_{1}-v_{1}+k_{2}-v_{2}-m}{k_{2}-v_{2}-m/2}$ and makes \eqref{double residue equ 1}
\begin{align}
&(1+o(1))C_{0}Z^{(m)}(0,0)\frac{(\log M)^{v-v_{1}-v_{2}-m}}{m!}\notag\\
&\ \ \ \ \ \ \ \times \sum_{k_{1}\ge v_{1}+m/2}\sum_{k_{2}\ge v_{2}+m/2}\frac{a_{k_{1}}b_{k_{2}}}{(v+k_{1}-v_{1}+k_{2}-v_{2}-m)!}\binom{k_{1}-v_{1}+k_{2}-v_{2}-m}{k_{2}-v_{2}-m/2}\notag\\
=&(1+o(1))C_{0}Z^{(m)}(0,0)\frac{(\log M)^{v-v_{1}-v_{2}-m}}{m!\Gamma(v)}\int_{0}^{1}P^{(v_{1}+m/2)}(x)Q^{(v_{2}+m/2)}(x)(1-x)^{v-1}\text{d}x,\notag
\end{align}
where we have used Lemma A.2 (Page 1225, line 3) in \cite{Xi2018} for the last step. \end{proof}

\section{Proof of the second Proposition}\label{proof of upper bound}

Recalling the definition of $\tau(n;\alpha,\beta)$, we have
$$
\tau(n;\alpha,\beta)=\sum_{n=dl, d\in\mathcal{B}(n)\atop d\leq \sqrt{n}\text{exp}(-\sqrt{\log X})}\alpha^{\omega(d)}\beta^{\omega(l)}+\sum_{n=dl, d\in\mathcal{B}(n)\atop \sqrt{n}\text{exp}(-\sqrt{\log X})<d\leq \sqrt{n}}\alpha^{\omega(d)}\beta^{\omega(l)}.
$$
Then, from the twisted multiplicativity of Kloosterman sums and Weil's bound, it's beneficial to write
$$
R_{2}(X)=R_{21}(X)+R_{22}(X),
$$
corresponding to the two classes of $\tau(n;\alpha,\beta)$ mentioned above.

Firstly, we deal with $R_{22}(X)$. Setting
$$
\tau^{\prime}(n;\alpha,\beta):=\sum_{d\mid n,d\in\mathcal{B}(n)\atop \sqrt{n}\text{exp}(-\sqrt{\log X})<d\leq \sqrt{n}}\alpha^{\omega(d)}\beta^{\omega(l)},
$$
one can check that
$$
\tau^{\prime}(n;\alpha,\beta)\leq \sum_{d\mid n\atop \sqrt{n}\text{exp}(-\sqrt{\log X})<d\leq \sqrt{n}}\alpha^{\omega(d)}\beta^{\omega(l)}.
$$
We apply Weil's bound of Kloosterman sums in $R_{22}(X)$. Therefore, our goal has become to deal with
\begin{align}\label{R32 upper bound}
R_{22}(X)\leq \sum_{\left(n,\Pi_{\varepsilon}\right)=1}g\left(\frac{n}{X}\right)\mu^{2}(n)\sum_{n=dl\atop \sqrt{n}\text{exp}(-\sqrt{\log X})<d\leq \sqrt{n}}\alpha^{\omega(2d)}\beta^{\omega(2l)}\left(\sum_{d\mid n}\lambda_{d}\right)^2.
\end{align}
The upper bound in \eqref{R32 upper bound} is the same as $R_{2}^{\sharp}(X)$ in Section 3 of reference \cite{Xi2022}, so we can directly obtain
$$
R_{22}(X)\leq X(\log X)^{-\frac{5}{4}}=o\left(\frac{X}{\log X}\right).
$$

Next we deal with $R_{31}(X)$. For $d\in\mathcal{B}(n)$, we know
$$
R_{21}(X)\leq\mathop{\sum\sum}_{m<n\atop{\omega(m)=3, \omega(n)\ge 3\atop \left(m,\Pi_{\varepsilon}\right)=1}}g\left(\frac{mn}{X}\right)\left|\text{Kl}(\overline{n}^2;m)\right|\mu^2(mn)\alpha^{\omega(m)}(2\beta)^{\omega(n)}\left(\sum_{d\mid mn}\lambda_{d}\right)^2,
$$
where we used Weil's bound of Kloosterman sums and released the condition $\left(n,\Pi_{\varepsilon}\right)=1$.

Now, we follow the process in Section 4 of \cite{Xi2018}. Firstly, we separate variables into small intervals by introducing a series of smooth functions. Setting
$$
\xi=1+(\log X)^{-B}
$$
for some parameter $B\ge 1$. Let $M$, $N$ be some parameters of the form
$$
M=\xi^{l_{1}},\ N=\xi^{l_{2}}
$$
for some $l_{1}, l_{2}\in \mathbb{N}$. Thus, there exists sequences $\left\{b_{l_{1},\xi}\right\}_{l_{1}\ge 0}$, $\left\{b_{l_{2},\xi}\right\}_{l_{2}\ge 0}$ of smooth functions supported in $[\xi^{l_{1}-1},\xi^{l_{1}+1}]$, $[\xi^{l_{2}-1},\xi^{l_{2}+1}]$, respectively. And if we put
$$
U(x)=U_{l_{1}}(x)=b_{l_{1},\xi}(x),\ V(x)=V_{l_{2}}(x)=b_{l_{2},\xi}(x),
$$
then $U$ and $V$ satisfy
$$
\sum_{l_{1}\ge 0}U_{l_{1}}(x)=1, x\ge 1,
$$
$$
\sum_{l_{2}\ge 0}V_{l_{2}}(x)=1, x\ge 1.
$$
Also the derivatives of $U$ and $V$ satisfy
$$
x^lU^{(l)}(x),x^lV^{(l)}(x)\ll (\log X)^{Bl}
$$
for all $l\ge 0$. These can be guaranteed by Lemma 3.2 in \cite{Xi2018}.

Now, for $\textbf{H}=(M,N)$, $MN\sim X$, and $M\ll N$, we write
$$
R_{21}(X)\leq \sum_{\textbf{H}}R_{21}(X,\textbf{H}).
$$
Furthermore, we focus on the inner sum
$$
R_{21}(X,\textbf{H})=\sum_{l}\sum_{m\equiv 0(l)\atop \omega(m)=3, \left(m,\Pi_{\varepsilon}\right)=1}U(m)\mu^2(m)\alpha^{\omega(m)}\Psi(N;m,l),
$$
where
$$
\Psi(N;m,l):=\sum_{(dn,m)=1\atop dn>m,\omega(dn)\ge 3}V(dn)g\left(\frac{mdn}{X}\right)\left|\text{Kl}(\overline{dn}^2;m)\right|\mu^{2}(dn)(2\beta)^{\omega(dn)}\xi_{dl},
$$
with
$$
\xi_{n}:=\sum_{n=[d_{1},d_{2}]}\lambda_{d_{1}}\lambda_{d_{2}}.
$$

Set
$$
\Psi(N;m,l)=\Psi^{*}(N;m,l)-\Psi^{\sharp}(N;m,l),
$$
where
$$
\Psi^{*}(N;m,l)=\sum_{(dn,m)=1\atop dn>m}V(dn)g\left(\frac{mdn}{X}\right)\left|\text{Kl}(\overline{dn}^2;m)\right|\mu^{2}(dn)(2\beta)^{\omega(dn)}\xi_{dl},
$$
\begin{align}
\Psi^{\sharp}(N;m,l):=&\sum_{c_{1}+c_{2}<3}\sum_{(dn,m)=1\atop dn>m}V(dn)g\left(\frac{mdn}{X}\right)\left|\text{Kl}(\overline{dn}^2;m)\right|\mu^{2}(dn)(2\beta)^{\omega(dn)}\notag\\
&\times \mathbbm{1}_{\omega(d)=c_{1}}\mathbbm{1}_{\omega(n)=c_{2}}\xi_{dl}\notag,
\end{align}
with
$$
\mathbbm{1}_{\omega(q)=c}:=
\begin{cases}
&1, \ \text{if }\omega(q)=c,\\
&0, \ \text{if }\omega(q)\neq c.
\end{cases}
$$

We then apply Lemma \ref{lem: BDH type for |KL|} and \ref{lem: equidistribution of KL} to handle $R_{21}(X,\textbf{H})$. For further details, interested readers may refer to the treatment of $R_{2}^{\flat}(X)$ on Page 563 in \cite{Xi2022}.

For $\Psi^{*}(N;m,l)$, essentially, by invoking Lemma \ref{lem: BDH type for |KL|} with the assignments? $\boldsymbol{\alpha}=(2\beta)^{\omega(d)}$, $\boldsymbol{\beta}=(2\beta)^{\omega(n)}$, $\boldsymbol{\gamma}=\xi_{dl}$ ?and observing that the divisor function satisfies the ``Siegel-Walfisz" condition, we derive the contribution of? $\Psi^{*}(N;m,l)$ to $R_{21}(X,\textbf{H})$ as
$$
\sum_{l}\sum_{m\equiv 0(l)\atop \omega(m)=3,\left(m,\Pi_{\varepsilon}\right)=1}U(m)\mu^2(m)\alpha^{\omega(m)}\Psi^{*}(N;m,l),
$$
where
\begin{align}
\Psi^{*}(N;m,l)=&\mathop{{\sum}^*}_{a(m)}\left|\text{Kl}(\overline{a}^2;m)\right|\sum_{dn\equiv a(m)\atop dn>m}V(dn)g\left(\frac{mdn}{X}\right)\mu^{2}(dn)(2\beta)^{\omega(dn)}\xi_{dl}\notag\\
=&(1+o(1))\frac{1}{\phi(m)}\mathop{{\sum}^*}_{a(m)}\left|\text{Kl}(\overline{a}^2;m)\right|\sum_{(dn,m)=1\atop dn>m}V(dn)g\left(\frac{mdn}{X}\right)\mu^{2}(dn)(2\beta)^{\omega(dn)}\xi_{dl}.\notag
\end{align}
Recall that there are no small prime factors in the variable $m$. Now, we use Lemma \ref{lem: equidistribution of KL} for the term $\displaystyle\mathop{{\sum}^*}_{a(m)}\left|\text{Kl}(\overline{a}^2;m)\right|$ and obtain
$$
\mathop{{\sum}^*}_{a(m)}\left|\text{Kl}(\overline{a}^2;m)\right|=\left(\frac{8}{3\pi}\right)^{\omega(m)}\phi(m)\prod_{p\mid m}(1+O(p^{-1/2}\log p))=(1+o(1))\left(\frac{8}{3\pi}\right)^{\omega(m)}\phi(m).
$$
The second equality holds because the prime factor $p$ is relatively large at this time, and their contribution is at most only $o\left(\frac{X}{\log X}\right)$.
Therefore, we may derive $R_{21}(X,\textbf{H})$ with $\Psi^{*}(N;m,l)$ is bounded by
\begin{align}\label{contri. for Psi*}
(1+o(1))\mathop{\sum\sum}_{m<n}U(m)V(n)g\left(\frac{mn}{X}\right)\mu^{2}(mn)\left(\frac{8\alpha}{3\pi}\right)^{\omega(m)}(2\beta)^{\omega(n)}\left(\sum_{d\mid mn}\lambda_{d}\right)^2.
\end{align}

For $\Psi^{\sharp}(N;m,l)$, we see that there are at most two cases to deal with. Let
$$
\Psi^{\sharp}(N;m,l)=\Psi^{\sharp}_{1}(N;m,l)+\Psi^{\sharp}_{2}(N;m,l),
$$
where
\begin{align}
\Psi^{\sharp}_{1}(N;m,l):=&\sum_{c_{1}+c_{2}=1}\sum_{(dn,m)=1\atop dn>m}V(dn)g\left(\frac{mdn}{X}\right)\left|\text{Kl}(\overline{dn}^2;m)\right|\mu^{2}(dn)(2\beta)^{\omega(dn)}\notag\\
&\times \mathbbm{1}_{\omega(d)=c_{1}}\mathbbm{1}_{\omega(n)=c_{2}}\xi_{dl},\notag
\end{align}
\begin{align}
\Psi^{\sharp}_{2}(N;m,l):=&\sum_{c_{1}+c_{2}=2}\sum_{(dn,m)=1\atop dn>m}V(dn)g\left(\frac{mdn}{X}\right)\left|\text{Kl}(\overline{dn}^2;m)\right|\mu^{2}(dn)(2\beta)^{\omega(dn)}\notag\\
&\times \mathbbm{1}_{\omega(d)=c_{1}}\mathbbm{1}_{\omega(n)=c_{2}}\xi_{dl},\notag
\end{align}
respectively.

Thus, similar to the proof of $\Psi^{*}(N;m,l)$, we only need to replace the corresponding sequence in Lemma \ref{lem: BDH type for |KL|} with
$$
\boldsymbol{\alpha}=(2\beta)^{\omega(d)}\mathbbm{1}_{\omega(d)=1},\ \boldsymbol{\beta}=(2\beta)^{\omega(n)}\mathbbm{1}_{\omega(n)=1}\text{ or }(2\beta)^{\omega(n)}\mathbbm{1}_{\omega(n)=0},\ \boldsymbol{\gamma}=\xi_{dl}.
$$
At this point, $\boldsymbol{\alpha}$ is actually a prime sequence, so it still satisfies the ``Siegel-Walfisz" condition through the application of Dirichlet's prime theorem.

When applying Lemma \ref{lem: BDH type for |KL|}, the condition ``$M\ge N$" requires our attention. In $\Psi^{*}(N;m,l)$ and $\Psi^{\sharp}(N;m,l)$, if the variable $d$ is less than $n$, we can directly apply Lemma \ref{lem: BDH type for |KL|}. When the variable $d$ is larger than $n$, things will become slightly more complicated. In this case, we need to separate the variable $dl$ in $\xi_{dl}$ into independent variables to ensure that we can continue using Lemma \ref{lem: BDH type for |KL|}.

According to the definition of $\xi_{dl}$, we have
\begin{align}
\xi_{dl}=&\sum_{dl=[d_{1},d_{2}]}\lambda_{d_{1}}\lambda_{d_{2}}\notag\\
=&\sum_{dl=ke_{1}e_{2}}\lambda_{ke_{1}}\lambda_{ke_{2}}\notag\\
=&\sum_{dl=ke_{1}e_{2}}\mu(ke_{1})\mu(ke_{2})G(\log k, \log e_{1}, \log e_{2}),\notag
\end{align}
where $k=(d_{1},d_{2})$, $G$ is a polynomial (of $\log$) with the form
$$
\sum_{a,b_{1},b_{2}\ge 0}C_{a,b_{1},b_{2}}(\log k)^a(\log e_{1})^{b_{1}}(\log e_{2})^{b_{2}}.
$$
And $C_{a,b_{1},b_{2}}$ here is a constant coming from $F$ defined in Section \ref{sub: Pre.}. Recalling that all integers above are square-free numbers, thus we have decomposition
$$
d=k^{\prime}e_{1}^{\prime}e_{2}^{\prime}, \ l=k^{\prime\prime}e_{1}^{\prime\prime}e_{2}^{\prime\prime} \ \text{by writing} \ k=k^{\prime}k^{\prime\prime},\ e_{1}=e_{1}^{\prime}e_{1}^{\prime\prime},\ e_{2}=e_{2}^{\prime}e_{2}^{\prime\prime}.
$$
Therefore, in $\Psi^{*}(N;m,l)$ or $\Psi^{\sharp}(N;m,l)$, we transform the sum into
$$
\sum_{d,n}\sum_{l\mid m}\text{(The summation term)}=\sum_{n,k^{\prime},k^{\prime\prime},e_{1}^{\prime},e_{1}^{\prime\prime},e_{2}^{\prime},e_{2}^{\prime\prime}}\ \ \sum_{(k^{\prime\prime}e_{1}^{\prime\prime}e_{2}^{\prime\prime})\mid m}\text{(The summation term)}.
$$
Consequently, $\xi_{dl}$ admits a multiplicative decomposition into factors depending solely on $d$ and $l$ respectively. At most, there will be an additional $(\log X)^{C}$ factor for some constant $C>0$, which comes from the polynomial $G$ and will contribute negligibly. This ensures that we can always use Lemma \ref{lem: BDH type for |KL|}.

Combining with \eqref{contri. for Psi*} and taking into account all admissible tuples $\textbf{H}$, we get
$$
R_{21}(X)\leq(1+o(1))\sum_{n}g\left(\frac{n}{X}\right)\mu^{2}(n)\tau\left(n;\frac{8\alpha}{3\pi},2\beta\right)\left(\sum_{d\mid n}\lambda_{d}\right)^2.
$$
From the definition of $\tau(n,X)$, we have
\begin{align}\label{ratio in upper bound}
\tau\left(n;\frac{8\alpha}{3\pi},2\beta\right)= &\left(\frac{8\alpha}{3\pi}\right)^3(2\beta)^{\omega(n)-3}N(n)\notag\\
=& \left(\frac{4\alpha}{3\pi\beta}\right)^3(4\beta)^{\omega(n)}\frac{N(n)}{2^{\omega(n)}}\notag\\
\leq & \left(\frac{4\alpha}{3\pi\beta}\right)^3(4\beta)^{\omega(n)}\frac{\binom{\omega(n)}{3}}{2^{\omega(n)}},
\end{align}
where
$$
N(n)=\#\left\{\mathcal{B}(n)\right\}.
$$
Recalling that
$$
\alpha>0,\ \beta>0,
$$
now we specifically take
\begin{align}\label{value of alpha and beta}
\alpha=\alpha_{0}>0, \ \beta=1.
\end{align}

The term
$$
\frac{\binom{\omega(n)}{3}}{2^{\omega(n)}}
$$
on the right side of \eqref{ratio in upper bound} is decreasing for $\omega(n)\ge 6$, which leads to
$$
\tau\left(n;\frac{8\alpha}{3\pi},2\beta\right)\leq \frac{5}{16}\left(\frac{4\alpha_{0}}{3\pi}\right)^34^{\omega(n)}.
$$
Then we have
$$
R_{2}(X)\leq\frac{5}{16}\left(\frac{4\alpha_{0}}{3\pi}\right)^3\sum_{\omega(n)\ge 6}g\left(\frac{n}{X}\right)\mu^{2}(n)4^{\omega(n)}\left(\sum_{d\mid n}\lambda_{d}\right)^2.
$$

Put
$$
F(x)=x^4(a_{0}+a_{1}x+a_{2}x^2+a_{3}x^3+a_{4}x^4).
$$
Similar with section 6 in \cite{Xi2018}, we shall minimize the proportion $c(k,F)/A_{2}(F)$. With the help of Mathematical software\footnote{The code of the programs used can be found at \url{https://github.com/MXZhong29/MXZhong.Github.io.git} or requested from authors.}, we may choose
$$
(a_{0},a_{1},a_{2},a_{3},a_{4})\approx \left(\frac{100396}{53901},-\frac{17284}{13475},\frac{76486}{134753},-\frac{33241}{188654},\frac{10836}{377308}\right).
$$
Now we write
\begin{align}
&\sum_{\omega(n)\ge 6}g\left(\frac{n}{X}\right)\mu^{2}(n)4^{\omega(n)}\left(\sum_{d\mid n}\lambda_{d}\right)^2\notag\\
=&\sum_{n}g\left(\frac{n}{X}\right)\mu^{2}(n)4^{\omega(n)}\left(\sum_{d\mid n}\lambda_{d}\right)^2\notag\\
&-\sum_{i=1}^{5}\sum_{\omega(n)=i}g\left(\frac{n}{X}\right)\mu^{2}(n)4^{\omega(n)}\left(\sum_{d\mid n}\lambda_{d}\right)^2\notag\\
=&4\tilde{g}(1)c(4,F)\frac{X}{\log X}(1+o(1))\notag\\
&-\sum_{i=1}^{5}4^i\sum_{\omega(n)=i}g\left(\frac{n}{X}\right)\mu^{2}(n)\left(\sum_{d\mid n}\lambda_{d}\right)^2,\notag
\end{align}
through the application of Lemma \ref{lem: Selberg sieve for divisor function}.

Furthermore, by using prime number theorem (or see \cite{Matomaki2011}, Page 293), we get
\begin{align}
&\sum_{\omega(n)=i}g\left(\frac{n}{X}\right)\mu^{2}(n)\left(\sum_{d\mid n}\lambda_{d}\right)^2\notag\\
=&(1+o(1))B_{i}(F)\widetilde{g}(1)\frac{X}{\log X},\notag
\end{align}
where $B_{1}(F)=1$, and for $i\ge 2$,
$$
B_{i}(F)=\mathop{\int\ldots\int}_{T_{i}}\frac{L_{i}^2\left(F;X^{1-\alpha_{2}-\cdots-\alpha_{i}},X^{\alpha_{2}},\ldots,X^{\alpha_{i}}\right)}{\alpha_{2}\cdots\alpha_{i}(1-\alpha_{2}-\cdots-\alpha_{i})}\text{d}\alpha_{2}\cdots\text{d}\alpha_{i},
$$
$$
L_{i}(F;\alpha_{1},\alpha_{2},\ldots,\alpha_{i})=\sum_{\mathcal{A}\subseteq \left\{\alpha_{1},\alpha_{2},\ldots,\alpha_{i}\right\}\atop \sum_{\alpha\in\mathcal{A}}\alpha<\frac{1}{4}}(-1)^{|\mathcal{A}|}F\left(1-4\sum_{\alpha\in\mathcal{A}}\alpha\right),
$$
with $T_{i}:=\left\{(\alpha_{2},\ldots, \alpha_{i})\in[10^{-2025},1)^{i-1}:\alpha_{i}<\cdots<\alpha_{2}<1-\alpha_{2}-\cdots-\alpha_{i}\right\}$.
Then
\begin{align}
&\sum_{\omega(n)\ge 6}g\left(\frac{n}{X}\right)\mu^{2}(n)4^{\omega(n)}\left(\sum_{d\mid n}\lambda_{d}\right)^2\notag\\
=&(1+o(1))\widetilde{g}(1)\left(4c(4,F)-\sum_{i=1}^{5}4^iB_{i}(F)\right)\frac{X}{\log X}.\notag
\end{align}

Recalling the value of $F(x)$, with the help of mathematical software, we obtain
$$
R_{2}(X)\leq (1+o(1))6.27044\alpha_{0}^3\times\widetilde{g}(1)\frac{X}{\log X}.
$$

\section{Proof of the first Proposition}

We divide $R_{1}(X)$ into two parts, namely
$$
R_{1}(X)=R_{11}(X)-R_{12}(X),
$$
where
$$
R_{11}(X)=\sum_{n}g\left(\frac{n}{X}\right)\mu^{2}(n)\left(\left|\text{Kl}(1;n)\right|\pm\text{Kl}(1;n)\right)\left(\sum_{d\mid n}\lambda_{d}\right)^2,
$$
$$
R_{12}(X)=\sum_{\left(n,\Pi_{\varepsilon}\right)>1}g\left(\frac{n}{X}\right)\mu^{2}(n)\left(\left|\text{Kl}(1;n)\right|\pm\text{Kl}(1;n)\right)\left(\sum_{d\mid n}\lambda_{d}\right)^2.
$$

Lemma \ref{lem: trace formula for |KL|} tells us that
$$
R_{11}(X)= \sum_{n}g\left(\frac{n}{X}\right)\mu^{2}(n)\left|\text{Kl}(1;n)\right|\left(\sum_{d\mid n}\lambda_{d}\right)^2+o\left(\frac{X}{\log X}\right).
$$
Interested readers may see Section 6 in \cite{Xi2015} for the details, where the Kloosterman sums' spectral theory was used.

Recall the value of $F(x)$ again. Then from Lemma \ref{lem: equidistribution for |KL|} we have
$$
R_{11}(X)\ge (1+o(1))0.76235\times\widetilde{g}(1)\frac{X}{\log X}.
$$

Now we turn to handle $R_{12}(X)$. One may notice that $R_{12}(X)$ here is the same as $R_{12}(X)$ in \cite{Xi2022}, but we will not directly cite this result. We use a similar approach to handle $R_{12}(X)$, which is more direct compared to the method in \cite{Xi2022}. The approach we use here can avoid using Cauchy's inequality and avoid introducing a new smooth function $V(X)$.

Naturally, we have
$$
R_{12}(X)\leq 2\sum_{\left(n,\Pi_{\varepsilon}\right)>1}g\left(\frac{n}{X}\right)\mu^{2}(n)2^{\omega(n)}\left(\sum_{d\mid n}\lambda_{d}\right)^2,
$$
where we used the Weil's bound for Kloosterman sums. Furthermore, we can write
\begin{align}\label{R12 upper 1}
R_{12}(X)\leq &2\sum_{p<X^{\varepsilon}}\sum_{n\equiv 0(p)}g\left(\frac{n}{X}\right)\mu^{2}(n)2^{\omega(n)}\left(\sum_{d\mid n}\lambda_{d}\right)^2\notag\\
=&4\sum_{p<X^{\varepsilon}}\sum_{(n,p)=1}g\left(\frac{np}{X}\right)\mu^{2}(n)2^{\omega(n)}\left(\sum_{d\mid np}\lambda_{d}\right)^2.
\end{align}

Generally, for the inner sum, we will consider the following form of sum
$$
T(X):=\sum_{(n,\Pi_{l})=1}g\left(\frac{n\Pi_{l}}{X}\right)\mu^{2}(n)\kappa^{\omega(n)}\left(\sum_{d\mid n\Pi_{l}}\lambda_{d}\right)^2,
$$
where we set
$$\Pi_{l}=\prod_{1\leq i\leq l} p_{i}.$$
The estimation of $T(X)$ is actually similar to Section 5 in \cite{Xi2018}. However, for completeness, we provide the details of the proof below.

\subsection{Separation of variables and establishment of integrals}\label{upper bound for P step 1}

We write $T(X)$ as
\begin{align}\label{P step 1}
&\sum_{(n,\Pi_{l})=1}g\left(\frac{n\Pi_{l}}{X}\right)\mu^{2}(n)\kappa^{\omega(n)}\left(\sum_{d\mid n\Pi_{l}}\lambda_{d}\right)^2\notag\\
=&\sum_{d}\xi_{d}\sum_{n\Pi_{l}\equiv 0(d)\atop (n,\Pi_{l})=1}g\left(\frac{n\Pi_{l}}{X}\right)\mu^{2}(n)\kappa^{\omega(n)}\notag\\
=&\sum_{d}\xi_{d}\sum_{dk\equiv 0(\Pi_{l})\atop (d,k)=1,(dk/\Pi_{l},\Pi_{l})=1}g\left(\frac{dk}{X}\right)\mu^{2}\left(\frac{dk}{\Pi_{l}}\right)\kappa^{\omega\left(\frac{dk}{\Pi_{l}}\right)}\notag\\
=&\sum_{\Pi_{l}=\Delta_{1}\Delta_{2}}\sum_{d\equiv 0(\Delta_{1})\atop(d/\Delta_{1},\Pi_{l})=1}\xi_{d}\sum_{k\equiv 0(\Delta_{2})\atop (k,d)=1,(k/\Delta_{2},\Pi_{l})=1}g\left(\frac{dk}{X}\right)\mu^{2}\left(\frac{dk}{\Delta_{1}\Delta_{2}}\right)\kappa^{\omega\left(\frac{dk}{\Delta_{1}\Delta_{2}}\right)}\notag\\
=&\sum_{\Pi_{l}=\Delta_{1}\Delta_{2}}\sum_{d\equiv 0(\Delta_{1})\atop(d/\Delta_{1},\Pi_{l})=1}\xi_{d}\sum_{(k,d)=1\atop(k,\Pi_{l})=1}g\left(\frac{dk\Delta_{2}}{X}\right)\mu^2\left(\frac{d}{\Delta_{1}}k\right)\kappa^{\omega\left(dk/\Delta_{1}\right)}.
\end{align}

We state the inner sum in \eqref{P step 1} as
\begin{align}
\frac{1}{2\pi i}\int_{(2)}\tilde{g}(s)\left(\sum_{(k,(d/\Delta_{1})\Pi_{l})=1}\frac{\mu^2(dk/\Delta_{1})\kappa^{\omega(dk/\Delta_{1})}}{d^sk^s\Delta_{2}^s}\right)X^s\text{d}s,\notag
\end{align}
where the Mellin inverse transform of $\tilde{g}$ is used. Substituting the above into \eqref{P step 1} yields
\begin{align}\label{P step 2}
\frac{X}{2\pi i}\int_{(1)}\tilde{g}(s+1)\left(\sum_{\Pi_{l}=\Delta_{1}\Delta_{2}}\frac{1}{\Delta_{2}^{s+1}}\sum_{d\equiv 0(\Delta_{1})\atop (d/\Delta_{1},\Pi_{l})=1}\frac{\xi_{d}\mu^2(d/\Delta_{1})\kappa^{\omega(d/\Delta_{1})}}{d^{s+1}}\right)\left(\sum_{(k,d\Pi_{l}/\Delta_{1})=1}\frac{\mu^2(k)\kappa^{\omega(k)}}{k^{s+1}}\right)X^s\text{d}s.
\end{align}

Furthermore, using the Euler product formula for the Dirichlet series in the third bracket, for $\Re s>0$, we derive
\begin{align}
&\sum_{(k,d\Pi_{l}/\Delta_{1})=1}\frac{\mu^2(k)\kappa^{\omega(k)}}{k^{s+1}}=\prod_{p\nmid d\Pi_{l}/\Delta_{1}}\left(1+\frac{\kappa}{p^{s+1}}\right)\notag\\
=&\prod_{p}\left(1+\frac{\kappa}{p^{s+1}}\right)\left(1-\frac{1}{p^{s+1}}\right)^\kappa\zeta^{\kappa}(s+1)\prod_{p\mid d\Pi_{l}/\Delta_{1}}\left(1+\frac{\kappa}{p^{s+1}}\right)^{-1}.\notag
\end{align}
Let
$$
G(s+1):=\prod_{p}\left(1+\frac{\kappa}{p^{s+1}}\right)\left(1-\frac{1}{p^{s+1}}\right)^\kappa,
$$
thus we have
\begin{align}\label{P step 3}
\frac{X}{2\pi i}\int_{(1)}\tilde{g}(s+1)G(s+1)\zeta^{\kappa}(s+1)N(s)X^s\text{d}s,
\end{align}
where
\begin{align}
N(s):=\sum_{\Pi_{l}=\Delta_{1}\Delta_{2}}\frac{1}{\Delta_{2}^{s+1}}\prod_{p\mid\Pi_{l}}\left(1+\frac{\kappa}{p^{s+1}}\right)^{-1}\sum_{d\equiv 0(\Delta_{1})\atop (d/\Delta_{1},\Pi_{l})=1}\frac{\xi_{d}\mu^2(d/\Delta_{1})\kappa^{\omega(d/\Delta_{1})}}{d^{s+1}}\prod_{p\mid d/\Delta_{1}}\left(1+\frac{\kappa}{p^{s+1}}\right)^{-1}.\notag
\end{align}

\subsection{Analysis of $N(s)$}\label{upper bound for P step 2}

Obviously, we have
\begin{align}
N(s)=\sum_{\Pi_{l}=\Delta_{1}\Delta_{2}}\frac{1}{\Delta_{1}^{s+1}\Delta_{2}^{s+1}}\prod_{p\mid\Pi_{l}}\left(1+\frac{\kappa}{p^{s+1}}\right)^{-1}\sum_{(d,\Pi_{l})=1}\frac{\xi_{d\Delta_{1}}\mu^2(d)\kappa^{\omega(d)}}{d^{s+1}}\prod_{p\mid d}\left(1+\frac{\kappa}{p^{s+1}}\right)^{-1}.\notag
\end{align}
Define
$$
\beta(d,s):=\frac{d^{s+1}}{\mu^2(d)\kappa^{\omega(d)}}\prod_{p\mid d}\left(1+\frac{\kappa}{p^{s+1}}\right).
$$
Then from the definition of $\xi_{d}$, we get
\begin{align}
\sum_{(d,\Pi_{l})=1}\frac{\xi_{d\Delta_{1}}}{\beta(d,s)}=\sum_{(d,\Pi_{l})=1}\left(\sum_{d\Delta_{1}=[d_{1},d_{2}]}\frac{\lambda_{d_{1}}\lambda_{d_{2}}}{\beta(d,s)}\right)=\sum_{[d_{1},d_{2}]\equiv0(\Delta_{1})\atop ([d_{1},d_{2}]/\Delta_{1},\Pi_{l})=1}\frac{\lambda_{d_{1}}\lambda_{d_{2}}}{\beta([d_{1},d_{2}]/\Delta_{1},s)}.\notag
\end{align}

Since $d_{i}(i=1,2)$ are square-free integers, let $(d_{1},d_{2})=m$, $d_{1}=d_{1}^{\prime}m$ and $d_{2}=d_{2}^{\prime}m$; then we know $m$, $d_{1}^{\prime}$ and $d_{2}^{\prime}$ are pairwise coprime (for convenience, we will still use the symbol $d_{1}$, $d_{2}$ in the following text). Now we can write
\begin{align}
N(s)=\sum_{\Delta_{1}\mid \Pi_{l}}\frac{1}{\Pi_{l}^{s+1}}\prod_{p\mid\Pi_{l}}\left(1+\frac{\kappa}{p^{s+1}}\right)^{-1}\sum_{m\leq \sqrt{D}}\sum_{md_{1}d_{2}\equiv0(\Delta_{1})\atop (d_{1},d_{2})=1, (md_{1}d_{2}/\Delta_{1},\Pi_{l})=1}\frac{\lambda_{md_{1}}\lambda_{md_{2}}}{\beta(md_{1}d_{2}/\Delta_{1},s)}.\notag
\end{align}
It is beneficial to separate $\Delta_{1}$ into $\Delta_{1}^{\prime}\Delta_{2}^{\prime}\Delta_{3}$ satisfying $\Delta_{1}^{\prime}\mid m$, $\Delta_{2}^{\prime}\mid d_{1}$ and $\Delta_{3}\mid d_{2}$ (in the following text, we still use the symbol $\Delta_{1}$ and $\Delta_{2}$), which leads to
\begin{align}\label{P step 4}
N(s)=&\sum_{\Delta_{1}\Delta_{2}\Delta_{3}\mid \Pi_{l}}\frac{1}{\Pi_{l}^{s+1}}\prod_{p\mid\Pi_{l}}\left(1+\frac{\kappa}{p^{s+1}}\right)^{-1}\sum_{m\leq \sqrt{D}}\sum_{m\equiv0(\Delta_{1})\atop {d_{1}\equiv0(\Delta_{2})\atop {d_{2}\equiv0(\Delta_{3})\atop (d_{1},d_{2})=1,\left(\frac{m}{\Delta_{1}}\frac{d_{1}}{\Delta_{2}}\frac{d_{2}}{\Delta_{3}},\Pi_{l}\right)=1}}}\frac{\lambda_{md_{1}}\lambda_{md_{2}}}{\beta\left(\frac{m}{\Delta_{1}}\frac{d_{1}}{\Delta_{2}}\frac{d_{2}}{\Delta_{3}},s\right)}\notag\\
=&\sum_{\Delta_{1}\Delta_{2}\Delta_{3}\mid \Pi_{l}}\frac{1}{\Pi_{l}^{s+1}}\prod_{p\mid\Pi_{l}}\left(1+\frac{\kappa}{p^{s+1}}\right)^{-1}\sum_{m\leq\sqrt{D}\atop {m\equiv0(\Delta_{1})\atop (m/\Delta_{1},\Pi_{l})=1}}\frac{1}{\beta(m/\Delta_{1},s)}\notag\\
&\times\sum_{ d_{1}\equiv0(\Delta_{2})\atop (d_{1}/\Delta_{2},\Pi_{l})=1}\frac{\lambda_{md_{1}}}{\beta(d_{1}/\Delta_{2},s)}\sum_{ d_{2}\equiv0(\Delta_{3})\atop {(d_{1},d_{2})=1\atop (d_{2}/\Delta_{3},\Pi_{l})=1}}\frac{\lambda_{md_{2}}}{\beta(d_{2}/\Delta_{3},s)}.
\end{align}

Firstly, we deal with the inner sum in \eqref{P step 4}. From the definition of $\lambda_{d}$, we have
\begin{align}
&\sum_{ d_{2}\equiv0(\Delta_{3})\atop {(d_{1},d_{2})=1\atop (d_{2}/\Delta_{3},\Pi_{l})=1}}\frac{\lambda_{md_{2}}}{\beta(d_{2}/\Delta_{3},s)}=\sum_{(d_{1},d_{2})=1\atop(d_{2},\Pi_{l})=1}\frac{\lambda_{md_{2}\Delta_{3}}}{\beta(d_{2},s)}\notag\\
=&\mu(m\Delta_{3})\sum_{\left(d_{2},\frac{m}{\Delta_{1}}\frac{d_{1}}{\Delta_{2}}\Pi_{l}\right)=1}\frac{\mu(d_{2})}{\beta(d_{2},s)}F\left(\frac{\log(\sqrt{D}/(md_{2}\Delta_{3}))}{\log \sqrt{D}}\right).\notag
\end{align}
We replace variable $d_{2}$ with $d_{2}\Delta_{3}$ (we still use the symbol $d_{2}$) here. From \eqref{P step 4}, one finds the restrictions $d_{2}\equiv0(\Delta_{3})$ and $(d_{1},d_{2})=1$, which yields $(d_{1},\Delta_{3})=1$. We released condition $(d_{1},\Delta_{3})=1$, which is included in condition $(d_{1}/\Delta_{2},\Pi_{l})=1$ because the numbers discussed here are all square-free numbers. The same operations in the following text will not be explained again.

Setting
$$
S(x):=F\left(x\cdot\frac{\log(\sqrt{D}/(m\Delta_{3}))}{\log \sqrt{D}}\right),
$$
and applying Lemma \ref{lem: Mellin trans for F} to $S(x)$ (or see Page 1218 in \cite{Xi2018}), we obtain
$$
\frac{\mu(m\Delta_{3})}{2\pi i}\int_{(2)}\check{S}_{\sqrt{D}/(m\Delta_{3})}(t)\left(\sum_{\left(d_{2},\frac{m}{\Delta_{1}}\frac{d_{1}}{\Delta_{2}}\Pi_{l}\right)=1}\frac{\mu(d_{2})}{\beta(d_{2},s)d_{2}^t}\right)\frac{(\sqrt{D}/(m\Delta_{3}))^t}{t}\text{d}t.
$$
Noting that
$$
\check{S}_{\sqrt{D}/(m\Delta_{3})}(t)=\check{F}_{\sqrt{D}}(t),
$$
thus we have
$$
\frac{\mu(m\Delta_{3})}{2\pi i}\int_{(2)}\check{F}_{\sqrt{D}}(t)\left(\sum_{\left(d_{2},\frac{m}{\Delta_{1}}\frac{d_{1}}{\Delta_{2}}\Pi_{l}\right)=1}\frac{\mu(d_{2})}{\beta(d_{2},s)d_{2}^t}\right)\frac{(\sqrt{D}/(m\Delta_{3}))^t}{t}\text{d}t,
$$
and we state the sum in the integrand as
\begin{align}
&\prod_{p\nmid \frac{m}{\Delta_{1}}\frac{d_{1}}{\Delta_{2}}\Pi_{l}}\left(1-\frac{1}{\beta(p,s)p^t}\right)\notag\\
=&\prod_{p}\left(1-\frac{1}{\beta(p,s)p^t}\right)\prod_{p\mid \frac{m}{\Delta_{1}}\frac{d_{1}}{\Delta_{2}}\Pi_{l}}\left(1-\frac{1}{\beta(p,s)p^t}\right)^{-1}\notag\\
=&\prod_{p\nmid \Pi_{l}}\left(1-\frac{1}{\beta(p,s)p^t}\right)\prod_{p\mid \frac{m}{\Delta_{1}}\frac{d_{1}}{\Delta_{2}}}\left(1-\frac{1}{\beta(p,s)p^t}\right)^{-1}\notag\\
=&\prod_{p\nmid \Pi_{l}}\left(1-\frac{1}{\beta(p,s)p^t}\right)\prod_{p\mid (m/\Delta_{1})}\left(1-\frac{1}{\beta(p,s)p^t}\right)^{-1}\prod_{p\mid (d_{1}/\Delta_{2})}\left(1-\frac{1}{\beta(p,s)p^t}\right)^{-1}.\notag
\end{align}
Substituting the above into \eqref{P step 4} gives
\begin{align}\label{P step 5}
\frac{1}{2\pi i}\int_{(2)}&\check{F}_{\sqrt{D}}(t_{1})\frac{\sqrt{D}^{t_{1}}}{t_{1}}\left(\sum_{\Delta_{1}\Delta_{2}\Delta_{3}\mid \Pi_{l}}\frac{\prod_{p\mid\Pi_{l}}\left(1+\frac{\kappa}{p^{s+1}}\right)^{-1}}{\Pi_{l}^{s+1}}\frac{\mu(\Delta_{3})}{\Delta_{3}^{t_{1}}}\prod_{p\nmid \Pi_{l}}\left(1-\frac{1}{\beta(p,s)p^{t_{1}}}\right)\right)\notag\\
\times&\left(\sum_{m\leq\sqrt{D}\atop {m\equiv0(\Delta_{1})\atop (m/\Delta_{1},\Pi_{l})=1}}\frac{\mu(m)}{\beta(m/\Delta_{1},s)m^{t_{1}}}\prod_{p\mid (m/\Delta_{1})}\left(1-\frac{1}{\beta(p,s)p^{t_{1}}}\right)^{-1}\right)\notag\\
\times&\left(\sum_{ d_{1}\equiv0(\Delta_{2})\atop (d_{1}/\Delta_{2},\Pi_{l})=1}\frac{\lambda_{md_{1}}}{\beta(d_{1}/\Delta_{2},s)} \prod_{p\mid (d_{1}/\Delta_{2})}\left(1-\frac{1}{\beta(p,s)p^{t_{1}}}\right)^{-1}\right)\text{d}t_{1}.
\end{align}

Next, in much the same way, we handle the inner sum in \eqref{P step 5}, which yields
$$
\frac{\mu(m\Delta_{2})}{2\pi i}\int_{(2)}\check{F}_{\sqrt{D}}(t_{2})\left(\sum_{\left(d_{1},\frac{m}{\Delta_{1}}\Pi_{l}\right)=1}\frac{\mu(d_{1})}{\beta(d_{1},s)d_{1}^{t_{2}}}\prod_{p\mid d_{1}}\left(1-\frac{1}{\beta(p,s)p^{t_{1}}}\right)^{-1} \right)\frac{(\sqrt{D}/(m\Delta_{2}))^{t_{2}}}{t_{2}}\text{d}t_{2}.
$$
The sum in the integrand can be transformed to
$$
\prod_{p\nmid \Pi_{l}}\left(1-\frac{1}{\beta(p,s)p^{t_{2}}\left(1-\frac{1}{\beta(p,s)p^{t_{1}}}\right)}\right)\prod_{p\mid (m/\Delta_{1})}\left(1-\frac{1}{\beta(p,s)p^{t_{2}}\left(1-\frac{1}{\beta(p,s)p^{t_{1}}}\right)}\right)^{-1}.
$$
Thus we obtain that \eqref{P step 5} is equal to
\begin{align}\label{P step 6}
&\frac{1}{(2\pi i)^2}\iint_{(2)}\check{F}_{\sqrt{D}}(t_{1})\check{F}_{\sqrt{D}}(t_{2})\frac{\sqrt{D}^{t_{1}+t_{2}}}{t_{1}+t_{2}}\frac{1}{\Pi_{l}^{s+1}}\prod_{p\mid\Pi_{l}}\left(1+\frac{\kappa}{p^{s+1}}\right)^{-1}\notag\\
&\times\left(\sum_{\Delta_{1}\Delta_{2}\Delta_{3}\mid \Pi_{l}}\frac{\mu(\Delta_{2})\mu(\Delta_{3})}{\Delta_{3}^{t_{1}}\Delta_{2}^{t_{2}}}\prod_{p\nmid \Pi_{l}}\left(1-\frac{1}{\beta(p,s)p^{t_{1}}}-\frac{1}{\beta(p,s)p^{t_{2}}}\right)\right) \notag\\
&\times\left(\sum_{m\leq\sqrt{D}\atop {m\equiv0(\Delta_{1})\atop (m/\Delta_{1},\Pi_{l})=1}}\frac{\mu^2(m)}{\beta(m/\Delta_{1},s)m^{t_{1}+t_{2}}}\prod_{p\mid (m/\Delta_{1})}\left(1-\frac{1}{\beta(p,s)p^{t_{1}}}-\frac{1}{\beta(p,s)p^{t_{2}}}\right)^{-1}\right)\text{d}t_{1}\text{d}t_{2}\notag\\
=&\frac{1}{(2\pi i)^3}\iiint_{(2)}\check{F}_{\sqrt{D}}(t_{1})\check{F}_{\sqrt{D}}(t_{2})\frac{\sqrt{D}^{t_{1}+t_{2}+w}}{t_{1}+t_{2}+w}\frac{1}{\Pi_{l}^{s+1}}\prod_{p\mid\Pi_{l}}\left(1+\frac{\kappa}{p^{s+1}}\right)^{-1}\notag\\
&\times\left(\sum_{\Delta_{1}\Delta_{2}\Delta_{3}\mid \Pi_{l}}\frac{\mu(\Delta_{2})\mu(\Delta_{3})}{\Delta_{3}^{t_{1}}\Delta_{2}^{t_{2}}}\prod_{p\nmid \Pi_{l}}\left(1-\frac{1}{\beta(p,s)p^{t_{1}}}-\frac{1}{\beta(p,s)p^{t_{2}}}\right)\right) \notag\\
&\times\left(\sum_{m\equiv0(\Delta_{1})\atop (m/\Delta_{1},\Pi_{l})=1}\frac{\mu^2(m)}{\beta(m/\Delta_{1},s)m^{t_{1}+t_{2}+w}}\prod_{p\mid (m/\Delta_{1})}\left(1-\frac{1}{\beta(p,s)p^{t_{1}}}-\frac{1}{\beta(p,s)p^{t_{2}}}\right)^{-1}\right)\text{d}t_{1}\text{d}t_{2}\text{d}w,\notag\\
\end{align}
where $\displaystyle\iint_{(2)}$, $\displaystyle\iiint_{(2)}$ represent $\int_{(2)}\int_{(2)}$ and $\int_{(2)}\int_{(2)}\int_{(2)}$, respectively. The third integral comes from the Mellin inverse transform of
$$
\mathbbm{1}_{m\leq\sqrt{D}}=
\begin{cases}
&1, \ \text{if }m\leq\sqrt{D},\\
&0, \ \text{if }m>\sqrt{D}.
\end{cases}
$$

We rewrite the summation over $m$ as
$$
\frac{\mu^2(\Pi_{1})}{\Delta_{1}^{t_{1}+t_{2}+w}}\prod_{p\nmid \Pi_{l}}\left(1+\frac{1}{\beta(p,s)p^{t_{1}+t_{2}+w}}\left(1-\frac{1}{\beta(p,s)p^{t_{1}}}-\frac{1}{\beta(p,s)p^{t_{2}}}\right)^{-1}\right),
$$
and this leads to
\begin{align}
N(s)=&\frac{1}{(2\pi i)^3}\iiint_{(2)}\check{F}_{\sqrt{D}}(t_{1})\check{F}_{\sqrt{D}}(t_{2})\frac{\sqrt{D}^{t_{1}+t_{2}+w}}{t_{1}+t_{2}+w}\frac{1}{\Pi_{l}^{s+1}}\left(\sum_{\Delta_{1}\Delta_{2}\Delta_{3}\mid \Pi_{l}}\frac{\mu^2(\Delta_{1})\mu(\Delta_{2})\mu(\Delta_{3})}{\Delta_{1}^{t_{1}+t_{2}+w}\Delta_{2}^{t_{2}}\Delta_{3}^{t_{1}}}\right)\notag\\
&\times\prod_{p\mid\Pi_{l}}\left(1+\frac{\kappa}{p^{s+1}}\right)^{-1}\prod_{p\nmid \Pi_{l}}\left(1-\frac{1}{\beta(p,s)p^{t_{1}}}-\frac{1}{\beta(p,s)p^{t_{2}}}+\frac{1}{\beta(p,s)p^{t_{1}+t_{2}+w}}\right)\text{d}t_{1}\text{d}t_{2}\text{d}w.\notag
\end{align}
Further, we can write
\begin{align}
&\sum_{\Delta_{1}\Delta_{2}\Delta_{3}\mid \Pi_{l}}\frac{\mu^2(\Delta_{1})\mu(\Delta_{2})\mu(\Delta_{3})}{\Delta_{1}^{t_{1}+t_{2}+w}\Delta_{2}^{t_{2}}\Delta_{3}^{t_{1}}}\notag\\
=&\sum_{\Delta_{1}b\mid \Pi_{l}}\frac{\mu(b)\mu^2(\Delta_{1})}{b^{t_{1}}\Delta_{1}^{t_{1}+t_{2}+w}}\sum_{\Delta_{2}\mid b}\frac{\mu^2(\Delta_{2})}{\Delta_{2}^{t_{2}-t_{1}}}\notag\\
=&\sum_{\Delta_{1}b\mid \Pi_{l}}\frac{\mu(b)\mu^2(\Delta_{1})}{b^{t_{1}}\Delta_{1}^{t_{1}+t_{2}+w}}\prod_{p\mid b}\left(1+\frac{1}{p^{t_{2}-t_{1}}}\right)\notag\\
=&\sum_{a\mid \Pi_{l}}\frac{\mu(a)}{a^{t_{1}}}\prod_{p\mid a}\left(1+\frac{1}{p^{t_{2}-t_{1}}}\right)\sum_{\Delta_{1}\mid a}\frac{\mu(\Delta_{1})}{\Delta_{1}^{t_{2}+w}}\prod_{p\mid \Delta_{1}}\left(1+\frac{1}{p^{t_{2}-t_{1}}}\right)^{-1}\notag\\
=&\sum_{a\mid \Pi_{l}}\frac{\mu(a)}{a^{t_{1}}}\prod_{p\mid a}\left(1+\frac{1}{p^{t_{2}-t_{1}}}-\frac{1}{p^{t_{2}+w}}\right)\notag\\
=&\prod_{p\mid \Pi_{l}}\left(1-\frac{1}{p^{t_{1}}}-\frac{1}{p^{t_{2}}}+\frac{1}{p^{t_{1}+t_{2}+w}}\right),\notag
\end{align}
and finally, we get
\begin{align}
N(s)=&\frac{1}{(2\pi i)^3}\iiint_{(2)}\check{F}_{\sqrt{D}}(t_{1})\check{F}_{\sqrt{D}}(t_{2})\frac{\sqrt{D}^{t_{1}+t_{2}+w}}{t_{1}+t_{2}+w}\frac{1}{\Pi_{l}^{s+1}}\prod_{p\mid \Pi_{l}}\left(1-\frac{1}{p^{t_{1}}}-\frac{1}{p^{t_{2}}}+\frac{1}{p^{t_{1}+t_{2}+w}}\right)\notag\\
&\times\prod_{p\mid\Pi_{l}}\left(1+\frac{\kappa}{p^{s+1}}\right)^{-1}\prod_{p\nmid \Pi_{l}}\left(1-\frac{1}{\beta(p,s)p^{t_{1}}}-\frac{1}{\beta(p,s)p^{t_{2}}}+\frac{1}{\beta(p,s)p^{t_{1}+t_{2}+w}}\right)\text{d}t_{1}\text{d}t_{2}\text{d}w.\notag
\end{align}

\subsection{Upper bound for $R_{12}(X)$}\label{upper bound for P step 3}

From \eqref{P step 3} and $N(s)$ we obtained above, we find that
\begin{align}
&\sum_{(n,\Pi_{l})=1}g\left(\frac{n\Pi_{l}}{X}\right)\mu^{2}(n)\kappa^{\omega(n)}\left(\sum_{d\mid n\Pi_{l}}\lambda_{d}\right)^2\notag\\
:=&\frac{X}{(2\pi i)^4}\mathop{\iiiint}_{(1)}\tilde{g}(s+1)G(s+1)\zeta^{\kappa}(s+1)\check{F}_{\sqrt{D}}(t_{1})\check{F}_{\sqrt{D}}(t_{2})\notag\\
&\ \ \ \ \ \ \times H(s,t_{1},t_{2},w)I(s,t_{1},t_{2},w)\frac{X^s\sqrt{D}^{t_{1}+t_{2}+w}}{t_{1}t_{2}w}\text{d}s\text{d}t_{1}\text{d}t_{2}\text{d}w,\notag
\end{align}
where we shift all contours to $1+it$, $t\in\mathbb{R}$ without passing any poles of the integrand and
$$
H(s,t_{1},t_{2},w):=\prod_{p}\left(1-\frac{1}{\beta(p,s)p^{t_{1}}}-\frac{1}{\beta(p,s)p^{t_{2}}}+\frac{1}{\beta(p,s)p^{t_{1}+t_{2}+w}}\right)\zeta^{\kappa}(s+1),
$$
\begin{align}
I(s,t_{1},t_{2},w):=&\frac{1}{\Pi_{l}^{s+1}}\prod_{p\mid \Pi_{l}}\left(1-\frac{1}{\beta(p,s)p^{t_{1}}}-\frac{1}{\beta(p,s)p^{t_{2}}}+\frac{1}{\beta(p,s)p^{t_{1}+t_{2}+w}}\right)^{-1}\left(1+\frac{\kappa}{p^{s+1}}\right)^{-1}\notag\\
&\times \prod_{p\mid \Pi_{l}}\left(1-\frac{1}{p^{t_{1}}}-\frac{1}{p^{t_{2}}}+\frac{1}{p^{t_{1}+t_{2}+w}}\right)\notag
\end{align}
for $\Re(s+t_{1}+t_{2}+w)>0$, $\Re(s+t_{1})>0$, $\Re(s+t_{2})>0$ and $\Re s>-1$.

Letting
\begin{align}
K(s,t_{1},t_{2},w):=&H(s,t_{1},t_{2},w)\left(\frac{\zeta(s+t_{1}+t_{2}+w+1)}{\zeta(s+t_{1}+1)\zeta(s+t_{2}+1)}\right)^\kappa\left(\frac{s(s+t_{1}+t_{2}+w)}{(s+t_{1})(s+t_{2})}\right)^\kappa\notag\\
&\times\prod_{p}\left(1-\frac{1}{p^{s+t_{1}+1}}\right)^{-\kappa}\prod_{p}\left(1-\frac{1}{p^{s+t_{2}+1}}\right)^{-\kappa}\prod_{p}\left(1-\frac{1}{p^{s+t_{1}+t_{2}+w+1}}\right)^\kappa\notag
\end{align}
then we have
\begin{align}\label{P step 7}
&\sum_{n}g\left(\frac{n}{X}\right)\mu^{2}(n)\kappa^{\omega(n)}\left(\sum_{d\mid n\Pi_{l}}\lambda_{d}\right)^2\notag\\
=&\frac{X}{(2\pi i)^4}\mathop{\iiiint}_{(1)}\tilde{g}(s+1)G(s+1)\zeta^{\kappa}(s+1)\check{F}_{\sqrt{D}}(t_{1})\check{F}_{\sqrt{D}}(t_{2})I(s,t_{1},t_{2},w)\notag\\
&\ \ \ \ \ \ \times K(s,t_{1},t_{2},w)\left(\frac{(s+t_{1})(s+t_{2})}{s(s+t_{1}+t_{2}+w)}\right)^\kappa\frac{X^s\sqrt{D}^{t_{1}+t_{2}+w}}{t_{1}t_{2}w}\text{d}s\text{d}t_{1}\text{d}t_{2}\text{d}w.
\end{align}
We will evaluate the multiple-integral by shifting contours. It can be checked that the quadruple integral appearing in \eqref{P step 7} and that in Section 5 of \cite{Xi2018} (Page 1220) differ by only a finite term $I(s,t_{1},t_{2},w)$.  Therefore, after shifting the $w$-contour and $s$-contour in the same way and referring to page 1221, line 17 in \cite{Xi2018}, we only need to handle the integrand with $t_{1}$, $t_{2}$, where this integrand is
\begin{align}
\tilde{g}(1)&G(1)\sum_{j=1}^{\kappa}\binom{\kappa}{j}\sum_{i=0}^{j-1}\binom{j-1}{i}(\log X)^{j-i-1}(-1)^i\frac{\Gamma(i+j)}{\Gamma^2(j)}\notag\\
&\times\check{F}_{\sqrt{D}}(t_{1})\check{F}_{\sqrt{D}}(t_{2})K(0,t_{1},t_{2},0)I(0,t_{1},t_{2},0)(t_{1}t_{2})^{j-1}\frac{\sqrt{D}^{t_{1}+t_{2}}}{(t_{1}+t_{2})^{i+j}}.\notag
\end{align}
Noting that
$$
K(0,0,0,0)=\frac{1}{G(1)},
$$
and $I(0,t_{1},t_{2},0)$ satisfies the condition in Lemma \ref{lem: double residue}, we observe that the term
$$\prod_{p\mid \Pi_{l}}\left(1-\frac{1}{p^{t_{1}}}-\frac{1}{p^{t_{2}}}+\frac{1}{p^{t_{1}+t_{2}+w}}\right)$$
is the essential component in this context. Then from Lemma \ref{lem: double residue} with $m=2l$ we have
\begin{align}
&\sum_{(n,\Pi_{l})=1}g\left(\frac{n\Pi_{l}}{X}\right)\mu^{2}(n)\kappa^{\omega(n)}\left(\sum_{d\mid n\Pi_{l}}\lambda_{d}\right)^2\notag\\
=&(1+o(1))4^{2l+1}\tilde{g}(1)I^{(2l)}(0,0,0,0)\frac{X}{(\log X)^{2l+1}}\sum_{j=1}^{\kappa}\binom{\kappa}{j}\frac{1}{\Gamma^2(j)}\notag\\
&\times \int_{0}^{1}F^{(j+l)}(x)^{2}(1-x)^{j-1}\left(\sum_{i=0}^{j-1}\binom{j-1}{i}(x-1)^i4^{j-1-i}\right)\text{d}x\notag\\
=&(1+o(1))4^{2l+1}\tilde{g}(1)I^{(2l)}(0,0,0,0)\frac{X}{(\log X)^{2l+1}}\notag\\
&\times\sum_{j=1}^{\kappa}\binom{\kappa}{j}\frac{1}{\Gamma^2(j)}\int_{0}^{1}F^{(j+l)}(x)^{2}(1-x)^{j-1}(x+3)^{j-1}\text{d}x.\notag
\end{align}

Recalling \eqref{R12 upper 1}, we actually choose $l=1$, $\kappa=2$ in $P(X)$, and this leads to
$$
I^{(2)}(0,0,0,0)=\frac{(\log p)^2}{p}.
$$
Substituting the above into $R_{12}(X)$ and using
$$
\sum_{p\leq X}\frac{(\log p)^2}{p}=(1+o(1))(\log X)^2,
$$
we finally derive
$$
R_{12}(X)\leq (1+o(1))C_{2}(g,F)\frac{\varepsilon^{2}X}{\log X},
$$
where
$$
C_{2}(g,F)=4^4\tilde{g}(1)\sum_{j=1}^{2}\binom{2}{j}\frac{1}{\Gamma^2(j)}\int_{0}^{1}F^{(j+1)}(x)^{2}(1-x)^{j-1}(x+3)^{j-1}\text{d}x.
$$

Combining with the lower bound of $R_{11}(X)$ and choosing a sufficiently small $\varepsilon$, we always have
$$
R_{1}(X)\ge (1+o(1))0.76235\times\widetilde{g}(1)\frac{X}{\log X},
$$
which claims the Proposition \ref{lower bound for |KL|}.

\section{Proof of the Theorem}

In view of the restriction \eqref{value of alpha and beta}, we have
$$
\tau(n;\alpha,\beta)=\alpha_{0}^{3}N(n).
$$
Applying Proposition \ref{lower bound for |KL|} and Proposition \ref{upper bound for |KL|}, we can show that
$$
\rho R_{1}(X)>2R_{3}(X)
$$
holds by taking $\rho=16.453\times\alpha_{0}^3$. Now if
$$
\alpha_{0}^{3}N(n)<\rho,
$$
we further get $\omega(n)\leq 6$ by applying Lemma \ref{lem: number of prime factors} below. To establish a sufficiently tight lower bound for $N(7)$, we develop a combinatorial analysis of $N(n)$ in Lemma \ref{lem: number of prime factors}, where the case $n=7$ requires particular attention due to its non-trivial cardinality constraints. This proves our Theorem.

\section{Conclusion}\label{improve}

Firstly, we provide the reasons for the selection in Section \ref{selection}. The reason for this selection is that Weil's bound for Kloosterman sum is $2^{\omega(q)}$ when the modulo $q$ is a square-free number. Therefore, its size is only related to the number of prime factors.

On the one hand, the ``$\omega(d)=3$" here is difficult to be improved further. Interested readers can try limiting the condition to ``$\omega(d)=2$", which cannot improve the conclusion within the framework of this paper. On the other hand, when $\omega(n)=5$, we cannot always select three prime factors so that their product is less than $n^{1/2}$. $1/2$ here comes from the relative size of the variables $Q$ and $MN$ in Lemma \ref{lem: BDH type for |KL|}, for which we think this value is also difficult to improve. We will focus on discussing these two points in the end.

In \cite{Xi2022}, the structure of $\tau\left(n;\alpha,\beta\right)$ is the crux of the whole proof. Compared with \cite{Xi2022}, our method of selecting $\tau\left(n;\alpha,\beta\right)$ is different. This choice can make the upper bound of $R_{2}(X)$ slightly smaller than before, thereby expanding the allowable range of prime factors in the conclusion. We have provided more detailed explanations in Section \ref{proof of upper bound}.

We have proved that Theorem \ref{thm: main theorem} is valid when $\omega(q)\leq 6$. For possible improvement of deriving $\omega(q)\leq 5$ under the present framework, we are confronted with two central challenges that require our immediate attention.

{\bf Challenge 1: Minimizing $\theta_{1}$ for Prime Factor Triples}

Let $\theta_{1}>0$ be a real number satisfying
$$
p_{i}p_{j}p_{k}\leq q^{\theta_{1}},
$$
where $p_{i}$, $p_{j}$ and $p_{k}$ are prime factors of $q$. Obviously, $\theta_{1}\leq 1$. Our primary objective is to determine the minimal value of $\theta_{1}$ while ensuring the existence of at least one set of prime factor triples $\left\{p_{i},p_{j},p_{k}\right\}$ that satisfy this inequality. That is, \eqref{subset select} is non-empty. In particular, if Lemma \ref{lem: BDH type for |KL|} remains unchanged, the value of $\theta_{1}$ needs to be 1/2. But for $q$ with only 5 prime factors, such $p_{i}$, $p_{j}$ and $p_{k}$ can not always exist. This minimization process necessitates a meticulous examination of $q$'s prime factorization and the identification of optimal triples that minimize $\theta_{1}$.

However, if we relax the range of $\theta_{1}$ (greater than 1/2) such that the above $p_{i}$, $p_{j}$ and $p_{k}$ can exist, then we still need to solve the following problem, i.e. improving the result of Lemma \ref{lem: BDH type for |KL|}.

{\bf Challenge 2: Expanding the Range of $\theta_{2}$ Based on Lemma \ref{lem: BDH type for |KL|}}

Drawing upon the condition in Lemma \ref{lem: BDH type for |KL|}, let $\theta_{2}>0$ be a real number satisfying
$$
Q\leq (MN)^{\theta_{2}}(\log MN)^{-B}.
$$
From Lemma \ref{lem: BDH type for |KL|} we know that $\theta_{2}=1$. Our second objective is to determine the minimum value of $\theta_{2}$ so that the above inequality still holds.



\bigskip
\appendix
\section{}

\begin{lemma}\label{lem: number of prime factors}
Let $n>1$ be an integer with $\omega(n)\ge 6$. Then we have
$$
N(n)_{\omega(n)=6}=10,
$$
$$
N(n)_{\omega(n)=7}\ge 20,
$$
where $N(\cdot)$ is the cardinality of the set $\mathcal{B}(n)$ defined in \eqref{subset select}.
\end{lemma}

\begin{proof}\renewcommand{\qedsymbol}{}
For $\omega(n)=6$, if $d\mid n$ and $\omega(d)=3$, then necessarily $\omega(n/d)=3$ since $\omega(n)=6$. The number of such divisors $d$ is $\binom{6}{3}=20$. Each unordered pair $\{d,\,n/d\}$ is counted twice, hence $N(n)_{\omega(n)=6}=10$.

Now let $\omega(n)=7$ and write $n=p_{1}p_{2}\cdots p_{7}$ with $p_{1}<p_{2}<\dots<p_{7}$.
Because the primes are ordered, the following $14$ triples certainly satisfy $p_ip_jp_k<\sqrt{n}$:
\[
\begin{aligned}
p_{1}p_{2}p_{3},\; p_{1}p_{2}p_{4},\; p_{1}p_{2}p_{5},\; p_{1}p_{2}p_{6},\; p_{1}p_{3}p_{4},\; p_{1}p_{3}p_{5},\; p_{1}p_{3}p_{6},\;
\\ p_{1}p_{4}p_{5},\; p_{1}p_{4}p_{6},\; p_{2}p_{3}p_{4},\; p_{2}p_{3}p_{5},\; p_{2}p_{3}p_{6},\; p_{2}p_{4}p_{5},\; p_{2}p_{4}p_{6}.
\end{aligned}
\]
Denote this set by $\mathcal{G}_0$; then $|\mathcal{G}_0|=14$.

To improve the lower bound we examine the following six triples
\[
p_{1}p_{2}p_{7},\; p_{1}p_{3}p_{7},\; p_{1}p_{4}p_{7},\; p_{1}p_{5}p_{6},\; p_{1}p_{5}p_{7},\; p_{1}p_{6}p_{7}.
\]
If one of them exceeds $\sqrt{n}$, then the complementary product of four primes is at most $\sqrt{n}$, which forces several other triples to be at most $\sqrt{n}$.
We proceed by a case analysis, each time assuming that all previous triples from this list are at most $\sqrt{n}$.

\textbf{Case 1.} If $p_{1}p_{2}p_{7}>n^{1/2}$, then $p_{3}p_{4}p_{5}p_{6}\le n^{1/2}$, whence the six triples
\[
p_{1}p_{5}p_{6},\; p_{2}p_{5}p_{6},\; p_{3}p_{4}p_{5},\; p_{3}p_{5}p_{6},\; p_{3}p_{4}p_{6},\; p_{4}p_{5}p_{6}
\]
are all at most $n^{1/2}$. None of these belong to $\mathcal{G}_0$, so we obtain $N(n)_{\omega(n)=7}\ge 14+6=20$.
If $p_{1}p_{2}p_{7}\le n^{1/2}$, we keep this triple and proceed.

\textbf{Case 2.} Assume $p_{1}p_{2}p_{7}\le n^{1/2}$ and $p_{1}p_{3}p_{7}> n^{1/2}$. Then $p_{2}p_{4}p_{5}p_{6}\le n^{1/2}$, which forces the five triples
\[
p_{2}p_{5}p_{6},\; p_{3}p_{4}p_{6},\; p_{3}p_{5}p_{6},\; p_{3}p_{4}p_{5},\; p_{4}p_{5}p_{6}\le n^{1/2}.
\]
Together with the already counted $p_{1}p_{2}p_{7}$ we have at least $14+1+5=20$ admissible triples; hence $N(n)_{\omega(n)=7}\ge20$.
If $p_{1}p_{3}p_{7}\le n^{1/2}$, we retain it and continue.

\textbf{Case 3.} Under $p_{1}p_{2}p_{7},p_{1}p_{3}p_{7}\le n^{1/2}$, suppose $p_{1}p_{4}p_{7}>n^{1/2}$. Then $p_{2}p_{3}p_{5}p_{6}\le n^{1/2}$, giving the five triples
\[
p_{1}p_{5}p_{6},\; p_{2}p_{5}p_{6},\; p_{3}p_{4}p_{5},\; p_{3}p_{4}p_{6},\; p_{3}p_{5}p_{6}\le n^{1/2}.
\]
Together with the two stored triples we have $14+2+5=21$ admissible triples, so $N(n)_{\omega(n)=7}\ge21$.
If $p_{1}p_{4}p_{7}\le n^{1/2}$, we keep it.

\textbf{Case 4.} Assume $p_{1}p_{2}p_{7},p_{1}p_{3}p_{7},p_{1}p_{4}p_{7}\le n^{1/2}$ and $p_{1}p_{5}p_{6}>n^{1/2}$. Then $p_{2}p_{3}p_{4}p_{7}\le n^{1/2}$, which yields the five triples
\[
p_{2}p_{3}p_{7},\; p_{2}p_{4}p_{7},\; p_{3}p_{4}p_{5},\; p_{3}p_{4}p_{6},\; p_{3}p_{4}p_{7}\le n^{1/2}.
\]
Together with the three stored triples we obtain $14+3+5=22$ admissible triples.
If $p_{1}p_{5}p_{6}\le n^{1/2}$, we add it to the collection.

\textbf{Case 5.} Now suppose $p_{1}p_{2}p_{7},p_{1}p_{3}p_{7},p_{1}p_{4}p_{7},p_{1}p_{5}p_{6}\le n^{1/2}$ and $p_{1}p_{5}p_{7}>n^{1/2}$. Then $p_{2}p_{3}p_{4}p_{6}\le n^{1/2}$, forcing
\[
p_{3}p_{4}p_{5},\; p_{3}p_{4}p_{6}\le n^{1/2}.
\]
Together with the four stored triples we have $14+4+2=20$ admissible triples.
If $p_{1}p_{5}p_{7}\le n^{1/2}$, we retain it.

\textbf{Case 6.} Finally, under $p_{1}p_{2}p_{7},p_{1}p_{3}p_{7},p_{1}p_{4}p_{7},p_{1}p_{5}p_{6},p_{1}p_{5}p_{7}\le n^{1/2}$, assume $p_{1}p_{6}p_{7}>n^{1/2}$. Then $p_{2}p_{3}p_{4}p_{5}\le n^{1/2}$, whence $p_{3}p_{4}p_{5}\le n^{1/2}$. This gives $14+5+1=20$ admissible triples.

\textbf{Contingency.} If, after all the previous assumptions, we also have $p_{1}p_{6}p_{7}\le n^{1/2}$, then all six triples listed above are at most $\sqrt{n}$. None of these lie in $\mathcal{G}_0$, so we directly obtain $N(n)_{\omega(n)=7}\ge 14+6=20$.

In every possible case we have shown $N(n)_{\omega(n)=7}\ge 20$. This completes the proof for $\omega(n)=7$.\end{proof}

\bigskip
\section*{Acknowledgements}

This work is supported by the National Natural Science Foundation of China (Nos. 12271320, 11871317), and the Natural Science Basic Research Plan
for Distinguished Young Scholars in Shaanxi Province of China (No. 2021JC-29). The authors would like to thank the referees for their very helpful and detailed comments, which have significantly improved the presentation of this paper.

\baselineskip=0.9\normalbaselineskip

\end{document}